\documentclass[final]{siamart190516}
\usepackage{amsfonts}
\usepackage{amsfonts,amsmath,amssymb}
\usepackage{mathrsfs,mathtools,stmaryrd,wasysym}
\usepackage{enumerate}
\usepackage{enumitem}
\usepackage{esint}
\usepackage{graphicx,psfrag}
\usepackage{hyperref}
\usepackage{amsmath,stackengine}
\usepackage{yfonts}
\usepackage{multirow} 
\DeclareMathAlphabet{\mathpzc}{OT1}{pzc}{m}{it}
\newcommand{\T}{\mathscr{T}}

\makeatletter
\renewcommand{\fnum@algorithm}{\ALG@name~\thealgorithm:}
\makeatother

\newsiamremark{remark}{Remark}

\newcommand{\TheTitle}{A nonlocal coupled system: analysis and discretization}
\newcommand{\ShortTitle}{A nonlocal coupled system}
\newcommand{\TheAuthors}{F. Bersetche, E. Ot\'arola, D. Quero}

\headers{\ShortTitle}{\TheAuthors}

\title{{\TheTitle}\thanks{FB is partially supported by ANPCyT through grant PICT 2018 - 3017 and Programa Regional MATH-AmSud AMSUD210013. EO is partially supported by Programa Regional MATH-AmSud  AMSUD210013.}}

\author{
Francisco Bersetche\thanks{Departamento de Matem\'atica, FCEyN, Universidad de Buenos Aires \& IMAS CONICET,
Pabell\'on I, Ciudad Universitaria 1428, Buenos Aires, Argentina. \email{fbersetche@dm.uba.ar}}
\and
Enrique Ot\'arola\thanks{Departamento de Matem\'atica, Universidad T\'ecnica Federico Santa Mar\'ia, Valpara\'iso, Chile. \email{enrique.otarola@usm.cl}}
\and
Daniel Quero\thanks{Departamento de Ciencias Exactas, Universidad de Los Lagos, Osorno, Chile \email{daniel.querotangol@ulagos.cl}}}

\ifpdf
\hypersetup{
  pdftitle={\TheTitle},
  pdfauthor={\TheAuthors}
}
\fi

\date{Draft version of \today.}

\begin{document}

\maketitle
\begin{abstract}
We analyze a nonlocal coupled system arising as the Euler--Lagrange equations of an energy functional involving regional fractional Laplacians of orders $s_1$ and $s_2$ ($ 0 < s_1,s_2 < 1$), each acting on a separate disjoint domain and coupled through a nonlocal interaction term depending on a kernel $J$. Under suitable assumptions on the domains and the kernel, we prove existence and uniqueness of the energy minimizer and derive regularity estimates in fractional Sobolev spaces. We introduce a finite element discretization and establish a priori error estimates. We develop an alternating Schwarz-type method for both the continuous and discrete problems and prove its geometric convergence. Numerical experiments validate the theoretical predictions and illustrate the performance of the method.
\end{abstract}


% REQUIRED

%%%%%%%%%%%%%%%%%%%%%%%%%%%%%%%%%%%%%%%%%%%%%%%%%%%%
%%%%%%%%%%%%%%%%%%%%%%%%%%%%%%%%%%%%%%%%%%%%%%%%%%%%
%%%%%%%%%%%%%%%%%%%%%%%%%%%%%%%%%%%%%%%%%%%%%%%%%%%%
%%%%%%%%%%%%%%%%%%%%%%%%%%%%%%%%%%%%%%%%%%%%%%%%%%%%

\section{Introduction}\label{sec:intro}

Let $\Omega_{1}$ and $\Omega_{2}$ be two open, connected, bounded domains with Lipschitz boundaries in $\mathbb{R}^{n}$ ($n \ge 2$) such that $\overline{\Omega}_{1} \cap \overline{\Omega}_{2} = \emptyset$, and let $d := \mathrm{dist}(\Omega_1,\Omega_2)$ denote the distance between $\Omega_{1}$ and $\Omega_{2}$. Note that $d > 0$. Let $s_1, s_2 \in (0,1)$. Let $J : \mathbb{R}^n \to \mathbb{R}$ be a nonnegative, symmetric, measurable function  such that $J \in L^{1}(\mathbb{R}^{n})$. We further assume that $\mathrm{supp}(J) \supset B_{r}(0)$ for some $r > d$,
as illustrated in Figure \ref{fig:NL_interaction}. In this work, we study the following \emph{nonlocal coupled system} and develop and analyze a numerical scheme for its approximation: Given $f_{1} : \Omega_1 \to \mathbb{R}$ and $f_{2} : \Omega_2 \to \mathbb{R}$, find $u := u_{1}\chi_{\Omega_{1}} + u_{2}\chi_{\Omega_{2}}$ such that
\begin{equation}
	\label{def:state_eq}
	\begin{aligned}
		2\mathcal{C}_{1} \int_{\Omega_2^c}\frac{u_{1}(x) - u_{1}(y)}{|x - y|^{n + 2s_{1}}}\,\mathrm{d}y
		+
		\int_{\Omega_2}
		J(x - y)\big(u_{1}(x) - u_{2}(y)\big)\,\mathrm{d}y
		&=
		f_{1}(x),
		\quad x \in \Omega_1,
		\\
		2\mathcal{C}_{2} \int_{\Omega_1^c}\frac{u_{2}(x) - u_{2}(y)}{|x - y|^{n + 2s_{2}}}\,\mathrm{d}y
		+
		\int_{\Omega_1}
		J(x - y)\big(u_{2}(x) - u_{1}(y)\big)\,\mathrm{d}y
		&=
		f_{2}(x),
		\quad x \in \Omega_2.
	\end{aligned}
\end{equation}
Here, for $i \in \{1,2\}$, we set $\mathcal{C}_{i} = c(n,s_i)/2$, where $c(n,s_i)$ denotes a positive normalization constant, and $\chi_{\Omega_{i}}$ denotes the characteristic function of the domain $\Omega_{i}$. We impose the volumetric boundary condition $u_i = 0$ in $\Omega_i^c$. In each equation of system \eqref{def:state_eq}, the first integral term  corresponds to a regional fractional Laplacian that captures the long-range diffusion within each domain, while the second integral accounts for the  interaction between the two components $u_1$ and $u_2$ across the gap between $\Omega_1$ and $\Omega_2$.

Classical partial differential equations (PDEs) have long served as the primary mathematical framework for modeling a wide range of scientific and technological processes. Such models are typically derived under smoothness assumptions and rely on local arguments, in the sense that the value of the solution at a given point is determined solely by its immediate neighborhood. Despite their remarkable success, classical PDEs fall short when describing phenomena involving singularities, discontinuities, or inherently long-range interactions. Such
phenomena commonly arise in fracture mechanics and continuum mechanics
\cite{MR2563154,MR1727557}, 
and in anomalous diffusion, including superdiffusion and subdiffusion arising in subsurface transport and turbulence \cite{2003WRR.39.1022S,Suzuki2023}. These limitations have motivated the development of nonlocal models \cite{doi:10.1137,MR3966542,MR3938295}, where the value of the governing operator at a given point depends on the behavior of the solution over an extended region of the domain.

Such nonlocal models find applications across a wide range of disciplines, including finance \cite{4fc30f5f-894d-3c6c-9b15-01f8f8d74820,MR2042661}, image science \cite{doi:10.1137/070698592,MR2608636}, and peridynamics \cite{doi:10.1137/110833294,MR4000179}. They are typically formulated as integro-differential equations whose defining feature is a kernel function that encodes long-range interactions. An important example from the family of nonlocal operators is the \emph{integral fractional Laplacian} $(-\Delta)^{s}$ ($0 < s < 1$), which corresponds to the infinitesimal generator of a stable Lévy process. In probabilistic terms, $(-\Delta)^{s}$ arises from a random motion allowing long jumps with a polynomially decaying tail; see \cite{Valdinoci} and \cite[Chapter 1]{MR3469920}. When the random walk is instead confined to a domain $U$, the resulting operator is the \emph{regional fractional Laplacian}, which integrates over $U$ rather than over the entire space. Under the volumetric boundary condition
$u_i = 0$ in $\Omega_i^c$, this is precisely the operator appearing in the first integral term of each equation in system \eqref{def:state_eq}.

Despite the success of nonlocal models in describing complex phenomena, many real-world systems involve regions where different physical regimes coexist and interact across distinct spatial domains. This has motivated the development of Local-to-Nonlocal (LtN) coupling methods, which combine the computational efficiency of classical PDE-based models with the accuracy of nonlocal formulations in regions where nonlocal effects are significant
\cite{marta}. Beyond computational convenience, LtN coupling methods also offer a principled way to circumvent the nontrivial task of prescribing nonlocal volumetric boundary conditions, since surface data are typically available and classical boundary conditions can be imposed in the local subdomain
\cite{marta,MR4778195}. A variety of approaches have been proposed in this direction, including optimization-based coupling \cite{MR3158780,MR3501316}, Arlequin-type methods \cite{MR2440614,MR2881373}, morphing techniques \cite{MR2910463,AZDOUD20131332}, blending mechanisms \cite{MR3600327,MR3641030}, splice methods \cite{MR4778195}, and varying-horizon approaches \cite{MR3600327}; we refer the reader to \cite{marta} for a comprehensive review.

In contrast to these LtN coupling strategies, the system \eqref{def:state_eq} studied in this work couples two purely nonlocal operators, each acting on one of the two disjoint domains, through a cross-domain interaction kernel $J$. To the best of our knowledge, this nonlocal-nonlocal coupling configuration on disjoint domains has not been previously explored in the literature. From a mathematical standpoint, the system \eqref{def:state_eq} poses several analytical challenges: the two regional fractional Laplacians act on separate domains with different fractional orders $s_1$ and $s_2$, and the coupling term introduces a nonlocal interaction across a positive-distance gap, requiring the development of a suitable functional framework. One possible physical interpretation is that of two populations of particles, each supported on one of the disjoint domains and undergoing anomalous diffusion via long-jump random walks, with the kernel $J$ regulating their mutual exchange across the gap.

The analysis of system \eqref{def:state_eq} gives rise to the following contributions:
\begin{itemize}[leftmargin=*]
\item[$\bullet$] \emph{Well-posedness via an energy functional}: We introduce a suitable energy functional $\mathcal{E}$ whose
Euler--Lagrange equations yield a weak formulation of
\eqref{def:state_eq}. We establish the coercivity, continuity, and strict convexity of $\mathcal{E}$ and apply the direct method of the calculus of variations to prove the existence and uniqueness of a minimizer $u \in H$, which is the unique solution of this weak formulation.

\item[$\bullet$] \emph{Weak formulation and regularity estimates}: We analyze
the bilinear form associated with the weak formulation of  \eqref{def:state_eq} and establish its continuity and  coercivity, yielding a stability estimate for the unique  solution $u \in H$. We further prove that $u$ enjoys
additional Sobolev regularity, which is the key ingredient in the derivation of the a priori error estimates.

\item[$\bullet$] \emph{Finite element discretization}: We propose a finite element discretization of \eqref{def:state_eq} based on piecewise linear elements and derive a priori error bounds in the energy norm $|\cdot|_H$ in terms of the mesh
parameters $h$ and $\mathfrak{h}$ and the regularity of the  solution.

\item[$\bullet$] \emph{Alternating scheme}: We propose and
analyze alternating schemes for both the continuous and
discrete problems, inspired by the classical Schwarz method
\cite{schwarz1870ueber,Lions}. We first establish geometric
convergence of the continuous scheme; the same arguments
then yield geometric convergence of the discrete scheme to
the unique solution of the discrete coupled system.

\item[$\bullet$] \emph{Numerical experiments}: We present two numerical experiments that validate the theoretical predictions and illustrate the performance of the devised finite element scheme.
\end{itemize}

The remainder of this work is organized as follows. Section~\ref{sec:notation_and_prel} introduces the notation,
assumptions, and functional framework. Section~\ref{sec:nonloc_coupled_model} establishes the well-posedness of \eqref{def:state_eq} via the minimization
of an energy functional $\mathcal{E}$, derives its Euler--Lagrange equations as the weak formulation of \eqref{def:state_eq}, and provides stability and regularity estimates. Section~\ref{sec:fem} presents the finite element discretization and derives a priori error estimates in the energy norm $|\cdot|_H$. Section~\ref{sec:altern_schemes}
proposes and analyzes continuous and discrete alternating schemes inspired by the classical Schwarz method and proves their geometric convergence. Finally, Section~\ref{sec:numerical_exp} presents two numerical experiments validating the theoretical predictions.

%%%%%%%%%%%%%%%%%%%%%%%%%%%%%%%%%%%%%%%%%%%%%%%%
%%%%%%%%%%%%%%%%%%%%%%%%%%%%%%%%%%%%%%%%%%%%%%%%
%%%%%%%%%%%%%%%%%%%%%%%%%%%%%%%%%%%%%%%%%%%%%%%%
%%%%%%%%%%%%%%%%%%%%%%%%%%%%%%%%%%%%%%%%%%%%%%%%
%%%%%%%%%%%%%%%%%%%%%%%%%%%%%%%%%%%%%%%%%%%%%%%%
%%%%%%%%%%%%%%%%%%%%%%%%%%%%%%%%%%%%%%%%%%%%%%%%

\section{Notation and preliminary remarks}
\label{sec:notation_and_prel}

We establish the notation and recall some facts that will be useful later.
 
\subsection{Notation}
Throughout this work, let $n \geq 2$. If $\Omega \subset \mathbb{R}^n$ is an open, bounded domain, we denote by $\partial \Omega$ the boundary of $\Omega$ and by $\Omega^{c}$ the complement of $\Omega$. Given $r>0$ and $x \in \mathbb{R}^n$, $B_r(x)$ denotes the open Euclidean ball of radius $r$ centered at $x$. We write $B_r = B_r(0)$ when $x=0$.

If $\mathscr{X}$ and $\mathscr{Y}$ are Banach function spaces,  $\mathscr{X}\hookrightarrow \mathscr{Y}$ denotes that $\mathscr{X}$ is continuously embedded in $\mathscr{Y}$. We denote by $\mathscr{X}'$ the dual of $\mathscr{X}$ and by $\| \cdot \|_{\mathscr{X}}$ the norm of $\mathscr{X}$. The duality pairing between $\mathscr{X}'$ and $\mathscr{X}$ is denoted by $\langle \cdot,\cdot \rangle_{\mathscr{X}',\mathscr{X}}$, and we write simply $\langle \cdot,\cdot \rangle$ when the spaces $\mathscr{X}'$ and $\mathscr{X}$ are clear from the context.

The relation $\mathfrak{a} \lesssim \mathfrak{b}$ indicates that $\mathfrak{a} \leq C \mathfrak{b}$ for some positive constant $C$
 independent of $\mathfrak{a}$, $\mathfrak{b}$, and of discretization parameters; $C$ may depend on $n$, $\Omega$, and any fractional exponents $s_i$ appearing in the model. The value of $C$ may change from line to line.

%%%%%%%%%%%%%%%%%%%%%%%%%%%%%%%%%%%%%%%%%%%%%%%%%%%%
%%%%%%%%%%%%%%%%%%%%%%%%%%%%%%%%%%%%%%%%%%%%%%%%%%%%
%%%%%%%%%%%%%%%%%%%%%%%%%%%%%%%%%%%%%%%%%%%%%%%%%%%%
%%%%%%%%%%%%%%%%%%%%%%%%%%%%%%%%%%%%%%%%%%%%%%%%%%%%

\subsection{Assumptions}\label{sec:assumptions}

Let $\Omega_{1}$ and $\Omega_{2}$ be two open, connected, bounded domains in $\mathbb{R}^{n}$. We assume that $\partial \Omega_1$ and $\partial \Omega_2$ are Lipschitz and $\bar{\Omega}_{1} \cap \bar{\Omega}_{2} = \emptyset$. The distance between the sets $\Omega_1$ and $\Omega_2$ is \cite[Section 0.6]{MR1681462}
\[
 d := \text{dist}(\Omega_{1},\Omega_{2}) := \inf\{|x - y|: x \in \Omega_{1},~y \in \Omega_{2}\}.
\]
Note that $d$ is a strictly positive constant because $\bar{\Omega}_{1} \cap \bar{\Omega}_{2} = \emptyset$.

For a domain $\Omega \subset \mathbb{R}^{n}$, we denote the distance from a point $x \in \mathbb{R}^{n}$ to $\Omega$ by \cite[Section 0.6]{MR1681462}
\begin{equation}
\delta_{\Omega}(x) = \inf\{|x - t|: t \in \Omega\}.
\end{equation}

We operate under the following assumptions on the kernel $J$:
\begin{itemize}
\item[(J1)] $J: \mathbb{R}^{n} \rightarrow \mathbb{R}$ is  nonnegative and symmetric, i.e., $J(z) \geq 0$ for a.e.~$z \in \mathbb{R}^n$ and $J(z) = J(-z)$ for a.e.~$z \in \mathbb{R}^n$;
\item[(J2)] $J \in L^{1}(\mathbb{R}^{n})$;
\item[(J3)] There exists $r > d$ such that $B_{r} \subset \text{supp}(J)$.
\end{itemize}

\begin{figure}[ht!]
\label{fig:NL_interaction}
\centering
\includegraphics[scale=0.35]{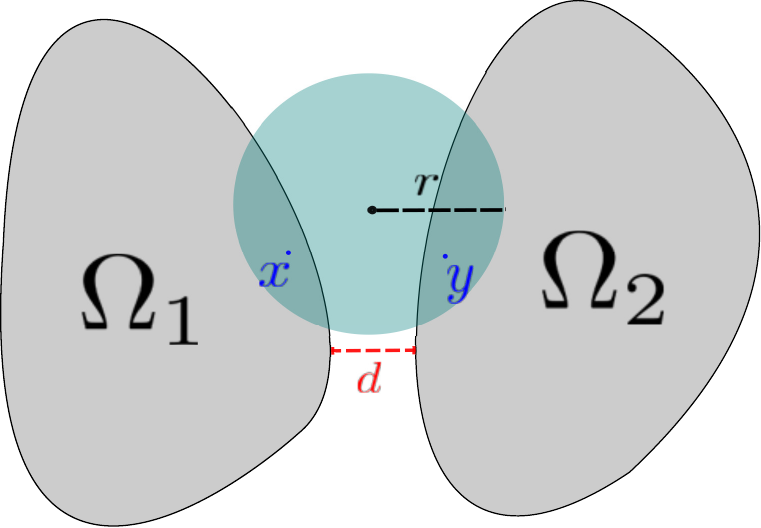}
\caption{Nonlocal interactions between $\Omega_{1}$ and $\Omega_{2}$ due to the kernel $J$.}
\end{figure}

\subsection{Convolution}
Let $f \in L^1(\mathbb{R}^n)$ and let $g \in L^p(\mathbb{R}^n)$ with $1 \leq p \leq \infty$. We define the convolution of $f$ with $g$ by
\begin{equation}
 (f \star g)(x) = \int_{\mathbb{R}^n}  f(x-y) g(y) \mathrm{d}y.
\end{equation}

The following is a classical result \cite[Theorem 4.15]{MR2759829}.

\begin{lemma}[continuity]\label{lemma:cont_conv}
If $f \in L^1(\mathbb{R}^n)$ and $g \in L^p(\mathbb{R}^n)$ with $1 \leq p \leq \infty$, then
\begin{equation}
\| f \star g \|_{L^p(\mathbb{R}^n)} \leq \| f \|_{L^1(\mathbb{R}^n)}\| g \|_{L^p(\mathbb{R}^n)}.
\end{equation}
\end{lemma}

%%%%%%%%%%%%%%%%%%%%%%%%%%%%%%%%%%%%%%%%%%%%%%%%%%%%
%%%%%%%%%%%%%%%%%%%%%%%%%%%%%%%%%%%%%%%%%%%%%%%%%%%%
%%%%%%%%%%%%%%%%%%%%%%%%%%%%%%%%%%%%%%%%%%%%%%%%%%%%
%%%%%%%%%%%%%%%%%%%%%%%%%%%%%%%%%%%%%%%%%%%%%%%%%%%%

\subsection{Function spaces}\label{sec:function_spaces}

Fractional Sobolev spaces provide a natural framework for analyzing problems involving fractional Laplace operators. In $\mathbb{R}^n$, a family of such spaces can be defined based on the Fourier transform $\mathcal{F}$: For $s \geq 0$, we define \cite[Definition 15.7]{MR2328004}, \cite[Chapter 1, Section 7]{MR0350178}
\begin{equation*}
H^{s}(\mathbb{R}^{n}) := \{ v \in L^2(\mathbb{R}^{n}) :  (1 + |\xi|^2)^{\frac{s}{2}} \mathcal{F}(v) \in  L^2(\mathbb{R}^{n})\},
\end{equation*}
endowed with the norm $\|v\|_{H^{s}(\mathbb{R}^{n})}:= \|(1 + |\xi|^2)^{\frac{s}{2}} \mathcal{F}(v)\|_{L^2(\mathbb{R}^{n})}$. 

Let $\Omega \subset \mathbb{R}^n$ be a bounded Lipschitz domain. We define $\tilde H^s(\Omega)$ as the closure of $C_0^{\infty}(\Omega)$ in $H^{s}(\mathbb{R}^{n})$ \cite[page 77]{MR1742312}. The space $\tilde H^s(\Omega)$ admits the following equivalent characterization as the space of zero-extension functions \cite[Theorem 3.29]{MR1742312}:
\begin{equation*}
 \tilde{H}^{s}(\Omega)=\{v|_{\Omega}^{} : v\in H^{s}(\mathbb{R}^{n}), \text{ supp }v\subset\bar{\Omega}\}.
\end{equation*}
For $s \in (0,1)$, we equip $\tilde{H}^{s}(\Omega)$ with the following inner product and seminorm \cite[page 75]{MR1742312}:
\begin{equation*}
\label{eq:inner_product}
( v, w )_{\tilde{H}^{s}(\Omega)} :=\int_{\mathbb{R}^{n}}\int_{\mathbb{R}^{n}}\frac{(v(x) - v(y))(w(x) - w(y))}{|x - y|^{n+2s}}\mathrm{d}x\mathrm{d}y,
\quad
| v |_{\tilde{H}^s(\Omega)}:= ( v, v )^{\frac{1}{2}}_{\tilde{H}^{s}(\Omega)}.
\end{equation*}
Owing to the fractional Poincar\'e inequality \cite[Proposition 2.4]{MR3620141}
\begin{equation}
 \| v \|_{L^2(\Omega)} \leq C | v |_{\tilde{H}^s(\Omega)}
 \quad
 \forall v \in \tilde{H}^s(\Omega),
 \qquad
 C = C(\Omega,n,s),
 \label{eq:Poincare}
\end{equation}
we have that $| \cdot |_{\tilde{H}^s(\Omega)}$ is a norm in $\tilde{H}^s(\Omega)$ and that $\tilde{H}^s(\Omega)$ is a Hilbert space.

We denote by $H^{-s}(\Omega)$ the dual space of $\tilde{H}^{s}(\Omega)$.

For $\mu \in (0,1)$, we introduce the Slobodecki\u i--Gagliardo seminorm \cite[(3.18)]{MR1742312}
\begin{equation}
 | v |_{H^{\mu}(\Omega)} = \left( \int_{\Omega}\int_{\Omega} \frac{|v(x) - v(y)|^2}{|x - y|^{n+2\mu}}\mathrm{d}x\mathrm{d}y\right)^{\frac{1}{2}}.
\end{equation}
We note that the definition of $| \cdot |_{H^{\mu}(\Omega)} $ remains valid for unbounded $\Omega \subseteq \mathbb{R}^n$. Let $s>0$ and write $s = r + \mu$ with $r \in \mathbb{N}_{0}$ and $\mu \in (0,1)$. We define \cite[page 74]{MR1742312}
\[
H^{s}(\Omega) := \{v \in H^{r}(\Omega): |\partial^{\alpha}v|_{H^{\mu}(\Omega)} < \infty \text{ for } |\alpha| = r\},
\]
equipped with the norm
\[
\|v\|_{H^{s}(\Omega)} := \left( \|v\|^{2}_{H^{r}(\Omega)} + \sum_{|\alpha| = r}|\partial^{\alpha}v|_{H^{\mu}(\Omega)}^{2} \right)^{\frac{1}{2}}.
\]

We conclude this section with the following Sobolev embedding results.

\begin{lemma}[embedding results]\label{lemma:embedding_result}
Let $s\in(0,1)$. If $\mathfrak{r}\in [1,2n/(n - 2s)]$, then $H^{s}(\Omega)\hookrightarrow L^{\mathfrak{r}}(\Omega)$. If $\mathfrak{r}\in [1,2n/(n - 2s))$, then $H^{s}(\Omega)\hookrightarrow L^{\mathfrak{r}}(\Omega)$ is compact.
\end{lemma}
\begin{proof}
We refer the reader to \cite[Theorem 6.7]{Hitchhikers} for a proof of $H^{s}(\Omega)\hookrightarrow L^{\mathfrak{r}}(\Omega)$. The fact that the embedding is compact for $\mathfrak{r}<2n/(n - 2s)$ follows from \cite[Corollary 7.2]{Hitchhikers}.
\end{proof}

%%%%%%%%%%%%%%%%%%%%%%%%%%%%%%%%%%%%%%%%%%%%%%%%%%%%
%%%%%%%%%%%%%%%%%%%%%%%%%%%%%%%%%%%%%%%%%%%%%%%%%%%%
%%%%%%%%%%%%%%%%%%%%%%%%%%%%%%%%%%%%%%%%%%%%%%%%%%%%
%%%%%%%%%%%%%%%%%%%%%%%%%%%%%%%%%%%%%%%%%%%%%%%%%%%%

\section{The nonlocal coupled model}\label{sec:nonloc_coupled_model}
Let $s_{1}, s_{2} \in (0,1)$. We define the space
\begin{equation}
H
: =
\tilde H^{s_1}(\Omega_1)\oplus \tilde H^{s_2}(\Omega_2)
= \{ u = u_1 + u_2: u_1 \in \tilde H^{s_1}(\Omega_1), u_2 \in \tilde H^{s_2}(\Omega_2) \}.
\label{eq:def_of_H}
\end{equation}
We equip the space $H$ with the following seminorm:
\begin{equation}
\label{eq:originalNorm}
| u |_{H} := \left( | u_1 |_{\tilde{H}^{s_1}(\Omega_1)}^2
+
| u_2 |_{\tilde{H}^{s_2}(\Omega_2)}^2 \right)^{\frac{1}{2}},
\qquad
u = u_1 + u_2. 
\end{equation}
Owing to the Poincar\'e inequality \eqref{eq:Poincare}, the seminorm $| \cdot |_{H}$ is a norm in $H$.

\subsection{The energy functionals}
We introduce two energy functionals $E$ and $\mathcal{E}$. As established in Section \ref{sec:euler_lagrange} below, the minimization of $\mathcal{E}$ leads to a weak formulation of the coupled system \eqref{def:state_eq}. First, we define $E: H \rightarrow [0,\infty)$ by
\begin{multline}
\label{eq:energy}
E(u) := \frac{\mathcal{C}_{1}}{2} \int_{\Omega_{2}^c}\int_{\Omega_{2}^c}\frac{(u_{1}(x) - u_{1}(y))^2}{|x - y|^{n + 2s_{1}}}\mathrm{d}y\mathrm{d}x
+
\frac{\mathcal{C}_{2}}{2} \int_{\Omega_{1}^c}\int_{\Omega_{1}^c}\frac{(u_{2}(x) - u_{2}(y))^2}{|x - y|^{n + 2s_{2}}}\mathrm{d}y\mathrm{d}x
\\
+ \dfrac{1}{2}\int_{\Omega_1}\int_{\Omega_{2}}J(x - y)(u_{1}(x) - u_{2}(y))^2 \mathrm{d}y\mathrm{d}x.
\end{multline}
Here, for $i \in \{1,2\}$, $\mathcal{C}_{i} = c(n,s_i)/2$, where $c(n,s_i)$ denotes a positive normalization constant.
The first two terms in \eqref{eq:energy} are well-defined by the definition of $H$ in \eqref{eq:def_of_H}, while the third term is well-defined because $J \in L^1(\mathbb{R}^n)$, $u_1 \in L^2(\Omega_1)$, and $u_2 \in L^2(\Omega_2)$. Since $J$ is a nonnegative function, $E(u) \geq 0$ for all $u \in H$. Second, given $f_1 \in L^2(\Omega_1)$ and $f_2 \in L^2(\Omega_2)$, we define $\mathcal{E}: H \rightarrow \mathbb{R}$ by
\begin{equation}
\label{eq:energy_2}
\mathcal{E}(u) := E(u) - \int_{\Omega_1} f_1 u_1 \mathrm{d}x - \int_{\Omega_2} f_2 u_2 \mathrm{d}x.
\end{equation}
Note that $\mathcal{E}$ is well-defined since $E$ is well-defined and $\int_{\Omega_1} f_1 u_1 \mathrm{d}x$ and $\int_{\Omega_2} f_2 u_2 \mathrm{d}x$ are finite by the Cauchy--Schwarz inequality.

\subsection{Properties of the energy functionals}
In this subsection, we examine some properties that the energy functionals $E$ and $\mathcal{E}$ satisfy. We begin with a lower bound for $E$ in $L^2$, which will be used to establish the coercivity of $\mathcal{E}$.

\begin{lemma}[A lower bound for $E$ in $L^2$]\label{lemma:L2_lower_bound_E}
There exists $\mathcal{C} > 0$ such that
\begin{equation}
E(u) \geq \mathcal{C} \left ( \|u_{1}\|^{2}_{L^{2}(\Omega_{1})} + \|u_{2}\|^{2}_{L^{2}(\Omega_{2})} \right) \quad \forall u \in H.
\end{equation}
\end{lemma}
\begin{proof}
We proceed by contradiction and assume that there is $\{ u_k \}_{k\in \mathbb{N}} \subset H$, so that for each $k \in \mathbb{N}$ we have
\begin{equation}
 \label{eq:contradiction_argument}
u_k = u_{1,k} + u_{2,k},
\qquad
\| u_{1,k} \|^2_{L^2(\Omega_1)}
+
\| u_{2,k} \|^2_{L^2(\Omega_2)} = 1,
\qquad
E(u_k) \leq k^{-1}.
% \quad
% \forall k \in \mathbb{N}.
\end{equation}
From \eqref{eq:contradiction_argument}, we can directly deduce that $E(u_k) \rightarrow 0$ as $k \uparrow \infty$.
% Without loss of generality, we can assume that $\{ u_{1,k} \}_{k\in \mathbb{N}}$ is such that $u_{1,k} \neq 0$ in $\Omega_1$ for every $k \in \mathbb{N}$.

As a first step, we note that since $u_{1,k}$ vanishes in $\Omega_1^c$, the first term in $E(u_k)$ equals the Slobodecki\u i--Gagliardo seminorm of $u_{1,k}$. Therefore, the definition of $E$ in \eqref{eq:energy} and the nonnegativity of $E$ allow us to conclude that, as $k \uparrow \infty$,
\begin{equation}
\int_{\Omega_{2}^c}\int_{\Omega_{2}^c}\frac{(u_{1,k}(x) - u_{1,k}(y))^2}{|x - y|^{n + 2s_{1}}}\mathrm{d}y\mathrm{d}x \rightarrow 0
\implies
| u_{1,k} |^2_{H^{s_1}(\Omega_1)}
% \int_{\Omega_{1}}\int_{\Omega_{1}}\frac{(u_{1,k}(x) - u_{1,k}(y))^2}{|x - y|^{n + 2s_{1}}}\mathrm{d}x\mathrm{d}y
\rightarrow 0.
\label{eq:convergence_in_semi_norm}
\end{equation}
This convergence result, together with the fact that $\| u_{1,k} \|_{L^2(\Omega_1)} \leq 1$ for all $k \in \mathbb{N}$, shows that the sequence $\{ u_{1,k} \}_{k \in \mathbb{N}}$ is uniformly bounded in $H^{s_1}(\Omega_1)$. Similar arguments show that $\{ u_{2,k} \}_{k \in \mathbb{N}}$ is uniformly bounded in $H^{s_2}(\Omega_2)$. We can thus extract nonrelabeled subsequences $\{ u_{1,k} \}_{k \in \mathbb{N}}$ and $\{ u_{2,k} \}_{k \in \mathbb{N}}$ such that
\begin{equation}
 \label{eq:nonrelabeled_subsequences_1}
\begin{aligned}
 u_{1,k} \rightharpoonup u_1~\mathrm{in}~H^{s_1}(\Omega_1),
 \quad
 u_{1,k} \rightarrow u_1~\mathrm{in}~L^2(\Omega_1),
 \quad
 k \uparrow \infty,
 \\
 u_{2,k} \rightharpoonup u_2~\mathrm{in}~H^{s_2}(\Omega_2),
 \quad
 u_{2,k} \rightarrow u_2~\mathrm{in}~L^2(\Omega_2),
 \quad
 k \uparrow \infty.
 \end{aligned}
\end{equation}
We now note that \eqref{eq:convergence_in_semi_norm} and a similar convergence argument for $\{ u_{2,k} \}_{k \in \mathbb{N}}$ show that the previously extracted subsequences satisfy the strong convergence properties:
\begin{equation}
 \label{eq:nonrelabeled_subsequences_2}
 u_{1,k} \rightarrow u_1~\mathrm{in}~H^{s_1}(\Omega_1),
 \quad
 u_{2,k} \rightarrow u_2~\mathrm{in}~H^{s_2}(\Omega_2),
 \quad
 k \uparrow \infty
\end{equation}
because weak convergence and convergence of norms imply strong convergence in $H^{s_1}(\Omega_1)$ and $H^{s_2}(\Omega_2)$. We again invoke \eqref{eq:convergence_in_semi_norm} and use a similar argument for  $\{ u_{2,k} \}_{k \in \mathbb{N}}$ to conclude that $|u_1|_{H^{s_1}(\Omega_1)} = |u_2|_{H^{s_2}(\Omega_2)} = 0$, and thus $u_1 = C_1$ and $u_2 = C_2$, where $C_1, C_2 \in \mathbb{R}$. Note that here we have used the fact $\Omega_1$ and $\Omega_2$ are connected.

On the other hand, since $E(u_{k}) \rightarrow 0$ as $k \uparrow \infty$, it also follows that
\[
\int_{\Omega_{1}}\int_{\Omega_{2}}J(x - y)(u_{1,k}(x) - u_{2,k}(y))^{2}\mathrm{d}y\mathrm{d}x \rightarrow 0, \quad k \uparrow \infty.
\]
This convergence property, along with the facts that $u_{1} = C_{1}$, $u_{2} = C_{2}$, the nonnegativity of $J$, and the property $\text{supp}(J) \supset B_{r}$, where $r > d = \textrm{dist}(\Omega_1,\Omega_2)$, allows us to conclude that
$$ (C_{1} - C_{2})^{2} \int_{\Omega_{1}}\int_{\Omega_{2}}J(x - y)\mathrm{d}y\mathrm{d}x = 0 \implies C_{1} = C_{2}. $$

As a final step, we show that $C_1 = C_2 = 0$. To do this, we proceed as follows. On one hand, we have that
\begin{multline}
 \int_{\Omega_{2}^c}\int_{\Omega_{2}^c}\frac{(u_{1,k}(x) - u_{1,k}(y))^2}{|x - y|^{n + 2s_{1}}}\mathrm{d}y\mathrm{d}x \rightarrow 0
 \\
 \implies
 \int_{\Omega_1}\int_{\Omega_{1}^c \cap \Omega_{2}^c}\frac{u_{1,k}(x)^2}{|x - y|^{n + 2s_{1}}}\mathrm{d}y\mathrm{d}x \rightarrow 0,
 \qquad k \uparrow \infty.
 \label{eq:auxiliary_convergence}
\end{multline}
We now introduce the set $\mathfrak{A} = \{ z \in \Omega_2^c : \textrm{dist}(z,\Omega_2) \leq d/2 \}$, which lies near $\Omega_2$ and is disjoint from $\Omega_1$. We note that
\[
 \int_{\Omega_1}\int_{\Omega_{1}^c \cap \Omega_{2}^c}\frac{u_{1,k}(x)^2}{|x - y|^{n + 2s_{1}}}\mathrm{d}y\mathrm{d}x
 \geq
 \int_{\Omega_1}\int_{\mathfrak{A}}\frac{u_{1,k}(x)^2}{|x - y|^{n + 2s_{1}}}\mathrm{d}y\mathrm{d}x.
\]
Define $h: \Omega_1 \rightarrow \mathbb{R}$ by $h(x) = \int_{\mathfrak{A}} |x - y|^{-n - 2s_{1}} \mathrm{d}y$. Note that $h \in L^{\infty}(\Omega_1)$. Indeed, since $d = \mathrm{dist}(\Omega_1,\Omega_2)$, we have that $|x - y| \geq d/2$ for $x \in \Omega_1$ and $y\in \mathfrak{A}$. Moreover, since $|x - y|$ is bounded above uniformly for $x \in \Omega_1$ and $y \in \mathfrak{A}$, the integrand $|x - y|^{-n - 2s_{1}}$ is bounded below by a positive constant on $\Omega_1 \times \mathfrak{A}$. Thus, since $|\mathfrak{A}|>0$, $h(x) \geq C_{\mathfrak{A}} > 0$ for all $x \in \Omega_1$. As a result,
\[
\int_{\Omega_1}\int_{\mathfrak{A}}\frac{u_{1,k}(x)^2}{|x - y|^{n + 2s_{1}}}\mathrm{d}y\mathrm{d}x
=
\int_{\Omega_1}u_{1,k}(x)^2 h(x)\mathrm{d}x
\geq
C_{\mathfrak{A}} \int_{\Omega_1}u_{1,k}(x)^2 \mathrm{d}x
\]
for every $k \in \mathbb{N}$. If we take the limit as $k \uparrow \infty$ in the previous inequality and use the strong convergence of $\{ u_{1,k} \}_{k\in\mathbb{N}}$ in $L^2(\Omega_1)$ as $k \uparrow \infty$ (see \eqref{eq:nonrelabeled_subsequences_1}), we obtain
\[
 C_{\mathfrak{A}} \int_{\Omega_1}u_{1}(x)^2 \mathrm{d}x \leq 0.
\]
Note that we have also used \eqref{eq:auxiliary_convergence}. Since $u_1 = C_1$, this shows that $C_{\mathfrak{A}}|\Omega_1| C_{1}^2 = 0$, which implies that $C_{1} = 0$. Therefore, $C_1 = C_2 = 0$, and $u_{1,k} \rightarrow 0$ in $L^2(\Omega_1)$ and $u_{2,k} \rightarrow 0$ in $L^2(\Omega_2)$ as $k\uparrow \infty$. This contradicts  \eqref{eq:contradiction_argument} and concludes the proof.
\end{proof}

\begin{lemma}[$\mathcal{E}$ is convex and continuous]\label{lemma:mathcalE_cont_convex}
The energy functional $\mathcal{E}$ defined in \eqref{eq:energy_2} is convex and continuous on $H$.
\end{lemma}
\begin{proof}
The fact that $\mathcal{E}$ is convex on $H$ is clear. Now, let us prove that $\mathcal{E}$ is continuous on $H$. To do this, let $\{u_{k}\}_{k \in \mathbb{N}} \subset H$ be such that
\begin{equation}\label{eq:mathcalE_cont_1}
u_{k} = u_{1,k} + u_{2,k},
\quad
u_{1,k} \rightarrow u_1~\mathrm{in}~\tilde{H}^{s_{1}}(\Omega_1),
\quad
u_{2,k} \rightarrow u_2~\mathrm{in}~\tilde{H}^{s_{2}}(\Omega_2), \quad k \uparrow \infty.
\end{equation}
In the following, we prove that $\mathcal{E}(u_{k}) \rightarrow \mathcal{E}(u)$ as $k \uparrow \infty$. From \eqref{eq:mathcalE_cont_1}, the Poincar\'e inequality \eqref{eq:Poincare}, and the fact that $f_1 \in L^2(\Omega_1)$ and $f_2 \in L^2(\Omega_2)$, it follows that
\[
\int_{\Omega_{1}}f_{1}u_{1,k} \mathrm{d}x \rightarrow \int_{\Omega_{1}}f_{1}u_{1} \mathrm{d}x, \qquad \int_{\Omega_{2}}f_{2}u_{2,k} \mathrm{d}x \rightarrow \int_{\Omega_{2}}f_{2}u_{2} \mathrm{d}x, \qquad k \uparrow \infty.
\]
It remains to prove that $E(u_{k}) \rightarrow E(u)$ as $k \uparrow \infty$. To do this, we begin with a straightforward application of the triangle inequality and obtain
\begin{align*}
\begin{split}
| |u_{1,k}|_{H^{s_{1}}(\Omega_{2}^{c})} - |u_{1}|_{H^{s_{1}}(\Omega_{2}^{c})} | \leq | u_{1,k} - u_{1} |_{H^{s_{1}}(\Omega_{2}^{c})} \leq | u_{1,k} - u_{1} |_{H^{s_{1}}(\mathbb{R}^{n})} \rightarrow 0,
\\
| |u_{2,k}|_{H^{s_{2}}(\Omega_{1}^{c})} - |u_{2}|_{H^{s_{2}}(\Omega_{1}^{c})} | \leq | u_{2,k} - u_{2} |_{H^{s_{2}}(\Omega_{1}^{c})} \leq | u_{2,k} - u_{2} |_{H^{s_{2}}(\mathbb{R}^{n})} \rightarrow 0
\end{split}
\end{align*}
as $k \uparrow \infty$. Recall that, for $v \in \tilde{H}^{s}(\Omega)$, $| v |_{\tilde{H}^s(\Omega)} = ( v, v )^{\frac{1}{2}}_{\tilde{H}^{s}(\Omega)} = ( v, v )^{\frac{1}{2}}_{H^{s}(\mathbb{R}^n)}$. We thus immediately obtain that
\begin{align}\label{eq:mathcalE_cont_2}
\begin{split}
\int_{\Omega_{2}^{c}}\int_{\Omega_{2}^{c}}\dfrac{|u_{1,k}(x) - u_{1,k}(y)|^{2}}{|x - y|^{n + 2s_{1}}}\mathrm{d}y\mathrm{d}x  & \rightarrow \int_{\Omega_{2}^{c}}\int_{\Omega_{2}^{c}}\dfrac{|u_{1}(x) - u_{1}(y)|^{2}}{|x - y|^{n + 2s_{1}}}\mathrm{d}y\mathrm{d}x,
\\
\int_{\Omega_{1}^{c}}\int_{\Omega_{1}^{c}}\dfrac{|u_{2,k}(x) - u_{2,k}(y)|^{2}}{|x - y|^{n + 2s_{2}}}\mathrm{d}y\mathrm{d}x  & \rightarrow \int_{\Omega_{1}^{c}}\int_{\Omega_{1}^{c}}\dfrac{|u_{2}(x) - u_{2}(y)|^{2}}{|x - y|^{n + 2s_{2}}}\mathrm{d}y\mathrm{d}x
\end{split}
\end{align}
as $k \uparrow \infty$. We still need to prove the convergence of the nonlocal term representing the coupling. As a first step, we write
\begin{multline*}
\int_{\Omega_{1}}\int_{\Omega_{2}}J(x - y)(u_{1,k}(x) - u_{2,k}(y))^{2}\mathrm{d}y\mathrm{d}x = \int_{\Omega_{1}}\int_{\Omega_{2}}J(x - y)u^{2}_{1,k}(x)\mathrm{d}y\mathrm{d}x
\\
- 2\int_{\Omega_{1}}\int_{\Omega_{2}}J(x - y)u_{1,k}(x)u_{2,k}(y)\mathrm{d}y\mathrm{d}x + \int_{\Omega_{1}}\int_{\Omega_{2}}J(x - y)u_{2,k}^{2}(y)\mathrm{d}y\mathrm{d}x := \mathrm{I}_{k} + \mathrm{II}_{k} + \mathrm{III}_{k}.
\end{multline*}
Let us analyze $\mathrm{I}_{k}$. For this purpose, we define $g: \Omega_{1} \rightarrow \mathbb{R}$ by $g(x):= \int_{\Omega_{2}}J(x - y)\mathrm{d}y$ for a.e.~$x \in \Omega_{1}$ and note that $g \in L^{\infty}(\Omega_{1})$. In fact, for a.e.~$x \in \Omega_{1}$, we have
\[
|g(x)| \leq \int_{\mathbb{R}^{n}}|J(x - y)| \mathrm{d}y = \|J\|_{L^{1}(\mathbb{R}^{n})}.
\]
We now note that the convergence property $u_{1,k} \rightarrow u_1$ in $L^2(\Omega_1)$ as $k \uparrow \infty$ guarantees that $\{ u_{1,k} \}_{k \in \mathbb{N}}$ is uniformly bounded in $L^2(\Omega_1)$ so there exists $\mathcal{M}>0$ such that $\|u_{1,k} + u_1 \|_{L^2(\Omega_1)} \leq \mathcal{M}$ for every $k \in \mathbb{N}$. This, together with $g \in L^{\infty}(\Omega_1)$, shows that
\begin{multline*}
\left| \int_{\Omega_1} g(x) (u^2_{1,k}(x) - u^2_{1}(x)) \mathrm{d}x \right|
\leq \| g \|_{L^{\infty}(\Omega_1)} \int_{\Omega_1} |u^2_{1,k}(x) - u_{1}^2(x)| \mathrm{d}x
\\
\leq \mathcal{M} \| g \|_{L^{\infty}(\Omega_1)} \left( \int_{\Omega_1} |u_{1,k}(x)- u_{1}(x)|^2 \mathrm{d}x \right)^{\frac{1}{2}} \rightarrow 0,
\qquad
k \uparrow \infty.
\end{multline*}
As a result,
\begin{align}\label{eq:mathcalE_cont_3}
\mathrm{I}_{k} = \int_{\Omega_{1}}u^{2}_{1,k}(x)g(x) \mathrm{d}x \rightarrow \int_{\Omega_{1}}\int_{\Omega_{2}}J(x - y)u_{1}^{2}(x) \mathrm{d}y\mathrm{d}x,
\qquad
k \uparrow \infty.
\end{align}
Similar arguments combined with Fubini's Theorem show that
\begin{align}\label{eq:mathcalE_cont_4}
\mathrm{III}_{k} \rightarrow \int_{\Omega_{1}}\int_{\Omega_{2}}J(x - y)u_{2}^{2}(y) \mathrm{d}y\mathrm{d}x, \qquad k \uparrow \infty.
\end{align}
It only remains to analyze $\mathrm{II}_{k}$. In the following, we prove that
\begin{equation}
\mathrm{II}_{k} \rightarrow \mathrm{II} := - 2  \int_{\Omega_{1}}\int_{\Omega_{2}}J(x - y)u_{1}(x)u_{2}(y) \mathrm{d}y\mathrm{d}x,
\qquad
k \uparrow \infty.
\label{eq:IIk--II}
\end{equation}
To this end, for each $k \in \mathbb{N}$, we define $\vartheta_{k}: \mathbb{R}^n \rightarrow \mathbb{R}$ by $\vartheta_{k}(x) = J \ast (u_{2,k} \chi_{\Omega_{2}})(x)$. We also define $\vartheta: \mathbb{R}^n \rightarrow \mathbb{R}$ by $\vartheta(x) = J \ast (u_{2} \chi_{\Omega_{2}})(x)$. We bound the difference $\vartheta - \vartheta_k$ in $L^2(\Omega_1)$ with the help of the continuity property of Lemma \ref{lemma:cont_conv}:
\[
\|\vartheta - \vartheta_{k}\|_{L^{2}(\Omega_{1})} \leq \|J\|_{L^{1}(\mathbb{R}^{n})}\| u_{2} - u_{2,k} \|_{L^{2}(\Omega_{2})} \rightarrow 0, \qquad k \uparrow \infty.
\]
In particular, $\{\vartheta_{k}\}_{k \in \mathbb{N}}$ is uniformly bounded in $L^{2}(\Omega_{1})$. Consequently,
\begin{align*}
& |\mathrm{II}_{k} - \mathrm{II}|
\leq
2 \int_{\Omega_{1}}|u_{1,k}(x) - u_{1}(x)||\vartheta_{k}(x)| \mathrm{d}x
+
2 \int_{\Omega_{1}}|u_{1}(x)||\vartheta_{k}(x) - \vartheta(x)| \mathrm{d}x
\\
& \leq 2 \|u_{1,k} - u_{1}\|_{L^{2}(\Omega_{1})}\|\vartheta_{k}\|_{L^{2}(\Omega_{1})} + 2 \| u_1 \|_{L^{2}(\Omega_{1})}\|\vartheta_{k} - \vartheta\|_{L^{2}(\Omega_{1})} \rightarrow 0, \qquad k \uparrow \infty.
\end{align*}
% Note that $\{u_{1,k}\}_{k \in \mathbb{N}}$ is uniformly bounded in $L^2(\Omega_1)$ because $u_{1,k} \rightarrow u_{1}$ in $\tilde{H}^{s_{1}}(\Omega_{1})$ as $k \uparrow \infty$.
Consequently, $\mathrm{II}_k \rightarrow \mathrm{II}$ as $k \uparrow \infty$, as we intended to show.
% \begin{align}\label{eq:mathcalE_cont_5}
% \mathbf{II}_{k} \rightarrow \int_{\Omega_{1}}\int_{\Omega_{2}}J(x - y)u_{1}(x)u_{2}(y) \mathrm{d}y\mathrm{d}x, \quad k \uparrow \infty.
% \end{align}

The convergence properties \eqref{eq:mathcalE_cont_2}--\eqref{eq:IIk--II} allow us to conclude that $E(u_{k}) \rightarrow E(u)$ as $k \uparrow \infty$. This concludes the proof.
\end{proof}

\subsection{Existence of a minimizer}

We are now ready to present one of the most important results in this section.

\begin{theorem}[existence of a unique minimizer]
Let $s_1,s_2 \in (0,1)$, let $f_{1} \in L^{2}(\Omega_{1})$, and let $f_2 \in L^{2}(\Omega_{2})$. Then, there exists a unique minimizer $u$ of $\mathcal{E}$ in $H$.
\label{thm:existence_of_minimizer}
 \end{theorem}
\begin{proof}
We proceed according to the direct method of the calculus of variations \cite[Theorem 5.51]{MR3026831}. To apply this method, we must verify that $\mathcal{E}$ is proper, convex, lower semicontinuous, and coercive in $H$. Since $\mathcal{E}$ is well-defined over $H$, it is clear that $\mathcal{E}$ is proper in $H$. The convexity and continuity of $\mathcal{E}$ follow from Lemma \ref{lemma:mathcalE_cont_convex}. It remains to prove that $\mathcal{E}$ is coercive in $H$.

Let $u = u_{1} + u_{2} \in H$, where $u_{1} \in \tilde{H}^{s_{1}}(\Omega_{1})$ and $u_{2} \in \tilde{H}^{s_{2}}(\Omega_{2})$. To prove the coercivity of $\mathcal{E}$, we first show that there exists a constant $C > 0$ such that
\begin{align}\label{eq:ineq_tildeHs_Hs_norm}
\begin{split}
\frac{\mathcal{C}_1}{2} \int_{\Omega_{2}^{c}}\int_{\Omega_{2}^{c}} \dfrac{|u_{1}(x) - u_{1}(y)|^{2}}{|x - y|^{n+2s_{1}}} \mathrm{d}y\mathrm{d}x 
+ 
\frac{\mathcal{C}_2}{2} \int_{\Omega_{1}^{c}}\int_{\Omega_{1}^{c}} \dfrac{|u_{2}(x) - u_{2}(y)|^{2}}{|x - y|^{n+2s_{2}}} \mathrm{d}y\mathrm{d}x
\\
\geq \frac{\mathcal{C}_1}{2}|u_1|_{\tilde{H}^{s_1}(\Omega_1)}^{2} + \frac{\mathcal{C}_2}{2}|u_2|_{\tilde{H}^{s_2}(\Omega_2)}^{2} - C \left ( \|u_{1}\|_{L^{2}(\Omega_{1})}^{2} + \|u_{2}\|_{L^{2}(\Omega_{2})}^{2} \right).
\end{split}
\end{align}
As a first step in deriving \eqref{eq:ineq_tildeHs_Hs_norm}, we use Fubini's theorem and the fact that $u_{1} = 0$ in $\Omega_{2}$ to obtain the following identity:
$$
|u_{1}|_{\tilde{H}^{s_{1}}(\Omega_{1})}^{2} = \int_{\Omega_{2}^{c}}\int_{\Omega_{2}^{c}} \dfrac{|u_{1}(x) - u_{1}(y)|^{2}}{|x - y|^{n+2s_{1}}} \mathrm{d}y\mathrm{d}x + 2\int_{\Omega_{2}^{c}}\int_{\Omega_{2}} \dfrac{|u_{1}(x) - u_{1}(y)|^{2}}{|x - y|^{n+2s_{1}}} \mathrm{d}y\mathrm{d}x.
$$
Let us now control the second term on the right-hand side of the previous expression. To this end, we first note that if $x \in \Omega_1$ and $y \in \Omega_2$, then $|x - y| \geq d = \mathrm{dist}(\Omega_1,\Omega_2)$. Exploiting further that $u_{1} = 0$ in $\Omega_{2}$, we write
\begin{align*}
\int_{\Omega_{2}^{c}}\int_{\Omega_{2}} \dfrac{|u_{1}(x) - u_{1}(y)|^{2}}{|x - y|^{n+2s_{1}}} \mathrm{d}y\mathrm{d}x
=
\int_{\Omega_1}\int_{\Omega_{2}} \dfrac{|u_{1}(x)|^{2}}{|x - y|^{n+2s_{1}}} \mathrm{d}y\mathrm{d}x
\leq |\Omega_{2}|d^{-n-2s_{1}}\|u_{1}\|^{2}_{L^{2}(\Omega_{1})}.
\end{align*}
As a result, we obtain the following bound
\begin{equation}\label{eq:first_ineq_tildeHs_Hs_norm}
|u_{1}|_{\tilde{H}^{s_{1}}(\Omega_{1})}^{2} - 2|\Omega_{2}|d^{-n-2s_{1}}\|u_{1}\|_{L^{2}(\Omega_{1})}^{2}
\leq \int_{\Omega_{2}^{c}}\int_{\Omega_{2}^{c}} \dfrac{|u_{1}(x) - u_{1}(y)|^{2}}{|x - y|^{n+2s_{1}}} \mathrm{d}y\mathrm{d}x.
\end{equation}
Similarly, the following estimate can be obtained
\begin{equation}\label{eq:second_ineq_tildeHs_Hs_norm}
|u_{2}|_{\tilde{H}^{s_2}(\Omega_{2})}^{2} - 2|\Omega_{1}|d^{-n-2s_{2}}\|u_{2}\|_{L^{2}(\Omega_{2})}^{2} \leq \int_{\Omega_{1}^{c}}\int_{\Omega_{1}^{c}} \dfrac{|u_{2}(x) - u_{2}(y)|^{2}}{|x - y|^{n+2s_{2}}} \mathrm{d}y\mathrm{d}x.
\end{equation}
If we add the inequalities in \eqref{eq:first_ineq_tildeHs_Hs_norm} and \eqref{eq:second_ineq_tildeHs_Hs_norm}, we obtain the desired estimate \eqref{eq:ineq_tildeHs_Hs_norm}. Using bound \eqref{eq:ineq_tildeHs_Hs_norm} and the nonnegativity of $J$, we deduce that
\[
E(u) \geq \frac{\mathcal{C}_1}{2}|u_1|_{\tilde{H}^{s_1}(\Omega_1)}^{2} + \frac{\mathcal{C}_2}{2}|u_2|_{\tilde{H}^{s_2}(\Omega_2)}^{2} - C \left( \|u_{1}\|_{L^{2}(\Omega_{1})}^{2} + \|u_{2}\|_{L^{2}(\Omega_{2})}^{2} \right).
\]
We can now apply the estimate from Lemma \ref{lemma:L2_lower_bound_E} and conclude that
\begin{equation}\label{eq:coercivity_E}
\mathfrak{C} E (u) \geq \frac{\mathcal{C}_1}{2}|u_1|_{\tilde{H}^{s_1}(\Omega_1)}^{2} + \frac{\mathcal{C}_2}{2}|u_2|_{\tilde{H}^{s_2}(\Omega_2)}^{2},
\end{equation}
where $\mathfrak{C} = 1 + \mathcal{C}^{-1}C$, $\mathcal{C}$ is the constant from Lemma \ref{lemma:L2_lower_bound_E}, and $C$ is the constant from \eqref{eq:ineq_tildeHs_Hs_norm}. Combining   \eqref{eq:coercivity_E} with the definition of $\mathcal{E}$ given in \eqref{eq:energy_2} and applying H\"older's and Young's inequalities, we obtain
\begin{multline*}
\mathfrak{C} \mathcal{E}(u) 
\geq 
\frac{\mathcal{C}_1}{2}|u_1|_{\tilde{H}^{s_1}(\Omega_1)}^{2}
+ 
\frac{\mathcal{C}_2}{2}|u_2|_{\tilde{H}^{s_2}(\Omega_2)}^{2}
\\
-
\mathfrak{C}\left( \epsilon_1 \|u_{1}\|_{L^{2}(\Omega_{1})}^{2} + \epsilon_2 \|u_{2}\|_{L^{2}(\Omega_{2})}^{2} 
+
\frac{1}{4\epsilon_1}\|f_{1}\|_{L^{2}(\Omega_{1})}^{2} 
+
\frac{1}{4\epsilon_2}\|f_{2}\|_{L^{2}(\Omega_{2})}^{2} \right),
\end{multline*}
where $\epsilon_1,\epsilon_2>0$. We now apply the Poincar\'e inequality \eqref{eq:Poincare} and choose $\epsilon_1$ and $\epsilon_2$ sufficiently small to obtain
\begin{equation*}
\mathcal{E}(u) 
\gtrsim
|u_1|_{\tilde{H}^{s_1}(\Omega_1)}^{2}
+ 
|u_2|_{\tilde{H}^{s_2}(\Omega_2)}^{2}
-
\|f_{1}\|_{L^{2}(\Omega_{1})}^{2} 
-
\|f_{2}\|_{L^{2}(\Omega_{2})}^{2}.
\end{equation*}
We can thus conclude that $\mathcal{E}$ is coercive in $H$; that is, $\mathcal{E}(u) \rightarrow \infty$ as $|u|_{H} \rightarrow \infty$.

After proving that $\mathcal{E}$ is coercive in $H$, the existence of a minimizer follows from the direct method of the calculus of variations. Uniqueness follows from the strict convexity of the functional $\mathcal{E}$. Since $\mathcal{E}$ is obtained from $E$ by subtracting linear terms, it suffices to note that $E$ is strictly convex. This follows from its quadratic structure.
\end{proof}

\subsection{The Euler--Lagrange equations}
\label{sec:euler_lagrange}

Let $u$ be the minimizer of $\mathcal{E}$ in $H$. The existence and uniqueness of this minimizer are guaranteed by Theorem \ref{thm:existence_of_minimizer}. We now derive the corresponding nonlocal coupled problem that determines $u$. Since $u$ is a minimizer of $\mathcal{E}$ in $H$, for every function $v \in H$, we have
\[
\dfrac{\partial}{\partial t}\mathcal{E}(u + tv)\Big|_{t = 0} = 0.
\]
A direct calculation shows that $u$ solves the weak formulation: Find $u \in H$ such that
\begin{equation}
\label{eq:euler_lagrange_eq}
\mathcal{A}(u,v) = \int_{\Omega_{1}}f_{1}(x)v_{1}(x) \mathrm{d}x + \int_{\Omega_{2}}f_{2}(x)v_{2}(x) \mathrm{d}x
\quad
\forall v \in H.
\end{equation}
Here, $\mathcal{A}: H \times H \rightarrow \mathbb{R}$ is defined as
\begin{multline}
 \mathcal{A}(u,v) = \mathcal{C}_{1}\int_{\Omega_{2}^{c}}\int_{\Omega_{2}^{c}}\dfrac{(u_{1}(x) - u_{1}(y))(v_{1}(x) - v_{1}(y))}{|x - y|^{n + 2s_{1}}}\mathrm{d}y\mathrm{d}x
 \\
 + \mathcal{C}_{2}\int_{\Omega_{1}^{c}}\int_{\Omega_{1}^{c}}\dfrac{(u_{2}(x) - u_{2}(y))(v_{2}(x) - v_{2}(y))}{|x - y|^{n + 2s_{2}}}\mathrm{d}y\mathrm{d}x
 \\
 + \int_{\Omega_{1}}\int_{\Omega_{2}}J(x - y)(u_{1}(x) - u_{2}(y))(v_{1}(x) - v_{2}(y))\mathrm{d}y\mathrm{d}x.
 \label{eq:bilinear_form}
\end{multline}

From the weak formulation \eqref{eq:euler_lagrange_eq}, we obtain a nonlocal coupled system as follows: First, we consider $v \in H$ with $v_2 = 0$ in \eqref{eq:euler_lagrange_eq} to obtain
\begin{align}\label{eq:first_EL_eq}
\begin{split}
& \mathcal{C}_{1}\int_{\Omega_{2}^{c}}\int_{\Omega_{2}^{c}}\dfrac{(u_{1}(x) - u_{1}(y))(v_{1}(x) - v_{1}(y))}{|x - y|^{n + 2s_{1}}}\mathrm{d}y\mathrm{d}x \\
& \quad + \int_{\Omega_{1}}\int_{\Omega_{2}}J(x - y)(u_{1}(x) - u_{2}(y))v_{1}(x)\mathrm{d}y\mathrm{d}x = \int_{\Omega_{1}}f_{1}(x)v_{1}(x) \mathrm{d}x
\end{split}
\end{align}
for all $v_1 \in \tilde{H}^{s_1}(\Omega_1)$. Second, we consider $v \in H$ with $v_{1} = 0$ in \eqref{eq:euler_lagrange_eq} to obtain
\begin{align}\label{eq:second_EL_eq}
\begin{split}
& \mathcal{C}_{2}\int_{\Omega_{1}^{c}}\int_{\Omega_{1}^{c}}\dfrac{(u_{2}(x) - u_{2}(y))(v_{2}(x) - v_{2}(y))}{|x - y|^{n + 2s_{2}}}\mathrm{d}y\mathrm{d}x \\
& \quad + \int_{\Omega_{1}}\int_{\Omega_{2}}J(x - y)(u_{2}(y) - u_{1}(x))v_{2}(y)\mathrm{d}y\mathrm{d}x = \int_{\Omega_{2}}f_{2}(x)v_{2}(x) \mathrm{d}x
\end{split}
\end{align}
for all $v_2 \in \tilde{H}^{s_2}(\Omega_2)$. 

\begin{remark}[equivalence and weak formulation] 
The following comments are appropriate.
\begin{enumerate}
\item The system \eqref{eq:first_EL_eq}--\eqref{eq:second_EL_eq} corresponds to a weak formulation of \eqref{def:state_eq}.
\item By construction, if $u = u_1 + u_2 \in H$ solves \eqref{eq:euler_lagrange_eq}, then $u_1$ and $u_2$ solve the nonlocal coupled problem \eqref{eq:first_EL_eq}--\eqref{eq:second_EL_eq}. Conversely, since $H = \tilde{H}^{s_1}(\Omega_1) \oplus \tilde{H}^{s_2}(\Omega_2)$, \eqref{eq:euler_lagrange_eq} can be obtained by adding \eqref{eq:first_EL_eq} and \eqref{eq:second_EL_eq}. Therefore, problem \eqref{eq:euler_lagrange_eq} and system \eqref{eq:first_EL_eq}--\eqref{eq:second_EL_eq} are equivalent.
\end{enumerate}
\end{remark}

Since it will be useful later, in the next lemma we present the most important properties that form $\mathcal{A}$ satisfies.

\begin{lemma}[properties of $\mathcal{A}$]\label{lemma:prop_A}
 $\mathcal{A}$ is bilinear, continuous, and coercive in $H \times H$, that is, there exist positive constants $M$ and $\alpha$ such that $\mathcal{A}(u,v) \leq M | u |_{H} | v |_{H}$ and $\mathcal{A}(v,v) \geq \alpha | v |^2_{H}$ for all $u,v \in H$.
\end{lemma}
\begin{proof}
We divide the proof into three steps.

\emph{Step 1.} The fact that $\mathcal{A}$ is bilinear is straightforward.

\emph{Step 2.} We now prove that $\mathcal{A}$ is bounded in $H \times H$. Let $u, v \in H$, and note that
\begin{align}
 \mathcal{C}_{1}\int_{\Omega_{2}^{c}}\int_{\Omega_{2}^{c}}\dfrac{(u_{1}(x) - u_{1}(y))(v_{1}(x) - v_{1}(y))}{|x - y|^{n + 2s_{1}}}\mathrm{d}y\mathrm{d}x
 \leq \mathcal{C}_1 | u_1 |_{\tilde{H}^{s_1}(\Omega_1)}| v_1 |_{\tilde{H}^{s_1}(\Omega_1)},
 \label{eq:continuity_C1}
 \\
  \mathcal{C}_{2}\int_{\Omega_{1}^{c}}\int_{\Omega_{1}^{c}}\dfrac{(u_{2}(x) - u_{2}(y))(v_{2}(x) - v_{2}(y))}{|x - y|^{n + 2s_{2}}}\mathrm{d}y\mathrm{d}x
 \leq \mathcal{C}_2 | u_2 |_{\tilde{H}^{s_2}(\Omega_2)}| v_2 |_{\tilde{H}^{s_2}(\Omega_2)},
 \label{eq:continuity_C2}
 \end{align}
which follow directly from the fact that $u_1$ and $v_1$ vanish outside $\Omega_1$, and $u_2$ and $v_2$ vanish outside $\Omega_2$, as well as from the Cauchy–Schwarz inequality. We now analyze the nonlocal term representing the coupling. To do this, we first bound it as
\begin{multline*}
\int_{\Omega_{1}}\int_{\Omega_{2}}J(x - y)(u_{1}(x) - u_{2}(y))(v_{1}(x) - v_{2}(y))\mathrm{d}y\mathrm{d}x
\\
\leq
\left[ \int_{\Omega_{1} \times \Omega_2} \! J(x - y)|u_{1}(x) - u_{2}(y)|^2\mathrm{d}y\mathrm{d}x \right]^{\frac{1}{2}}
\left[ \int_{\Omega_{1} \times \Omega_2} \! J(x - y)|v_{1}(x) - v_{2}(y)|^2\mathrm{d}y\mathrm{d}x \right]^{\frac{1}{2}}.
\end{multline*}
Let us now write the square of the first term on the right-hand side of the previous bound as follows; the other term can be treated similarly:
\begin{multline}
\int_{\Omega_{1}} \int_{\Omega_2} J(x - y)(u_{1}(x) - u_{2}(y))^2\mathrm{d}y\mathrm{d}x
=
\int_{\Omega_{1}} \int_{\Omega_2} J(x - y)u^2_{1}(x) \mathrm{d}y\mathrm{d}x
\\
- 2 \int_{\Omega_{1}} \int_{\Omega_2} J(x - y)u_{1}(x)u_{2}(y) \mathrm{d}y\mathrm{d}x
+
\int_{\Omega_{1}} \int_{\Omega_2} J(x - y)u^2_{2}(y) \mathrm{d}y\mathrm{d}x
 := \mathrm{I}_1 -2 \mathrm{I}_2 + \mathrm{I}_3.
\end{multline}
The control of $\mathrm{I}_1$, $\mathrm{I}_2$, and $\mathrm{I}_3$ follows from the arguments in the proof of Lemma \ref{lemma:mathcalE_cont_convex}. Below, we briefly present these arguments. First, we have the bounds
\begin{equation}
 \mathrm{I}_1 \leq \| J \|_{L^1(\mathbb{R}^n)} \| u_1 \|^2_{L^2(\Omega_1)},
 \qquad
 \mathrm{I}_3 \leq  \| J \|_{L^1(\mathbb{R}^n)} \| u_2 \|^2_{L^2(\Omega_2)}.
 \label{eq:I_1_and_I_3}
\end{equation}
Second, using the definition of convolution, the Cauchy–Schwarz inequality, the estimate from Lemma \ref{lemma:cont_conv}, and Young’s inequality, we obtain
\begin{multline}
 \mathrm{I}_2 = \int_{\Omega_1} u_1(x) (J \star u_2 \chi_{\Omega_2})(x) \mathrm{d}x
 \leq \| u_1 \|_{L^2(\Omega_1)} \| J \star u_2 \chi_{\Omega_2} \|_{L^2(\mathbb{R}^n)}
 \\
 \leq \| u_1 \|_{L^2(\Omega_1)} \| J  \|_{L^1(\mathbb{R}^n)} \|  u_2  \|_{L^2(\Omega_2)}
 \leq \frac{\| J  \|_{L^1(\mathbb{R}^n)} }{2} \left[ \| u_1 \|^2_{L^2(\Omega_1)} + \|  u_2  \|^2_{L^2(\Omega_2)} \right].
 \label{eq:I_2}
\end{multline}
A collection of bounds \eqref{eq:continuity_C1}, \eqref{eq:continuity_C2}, \eqref{eq:I_1_and_I_3}, and \eqref{eq:I_2}, and their analogues for $v$,
and a suitable use of the Poincar\'e inequality \eqref{eq:Poincare} establishes the continuity of $\mathcal{A}$:
\begin{equation}
\mathcal{A}(u,v) \leq C |u|_H |v|_H \quad
\forall
u,v \in H.
\end{equation}

\emph{Step 3.} Finally, we prove that the bilinear form $\mathcal{A}$ is coercive in $H \times H$. Given $u \in H$, we observe from the definitions of $E$ and $\mathcal{A}$ given in \eqref{eq:energy} and \eqref{eq:bilinear_form}, respectively, that $\mathcal{A}(u,u) = 2E(u)$. We then use the bound \eqref{eq:coercivity_E} to obtain
\begin{equation}
\mathcal{A}(u,u) \geq \alpha \left( |u_1|_{\tilde{H}^{s_1}(\Omega_1)}^{2} + |u_2|_{\tilde{H}^{s_2}(\Omega_2)}^{2} \right) = \alpha |u|_{H}^{2}.
\label{eq:A_is_coercive}
\end{equation}

We have thus proved that $\mathcal{A}$ is bilinear, continuous, and coercive in $H \times H$. This concludes the proof.
\end{proof}

\begin{remark}[inner product]
Since $\mathcal{A}$ is symmetric in $H \times H$,  Lemma \ref{lemma:prop_A} guarantees that $\mathcal{A}$ defines an inner product on $H$.
\label{rem:inner_product}
\end{remark}

We conclude this section with the following stability bound.

\begin{theorem}[stability bound]\label{thm:stab_bound}
The solution $u = u_1 + u_2 \in \tilde{H}^{s_1}(\Omega_1) \oplus \tilde{H}^{s_2}(\Omega_2)$ of \eqref{eq:euler_lagrange_eq}, or equivalently of the coupled system \eqref{eq:first_EL_eq} and \eqref{eq:second_EL_eq}, satisfies the following stability bound 
\begin{equation}
 |u_1|_{\tilde{H}^{s_1}(\Omega_1)}
 +
 |u_2|_{\tilde{H}^{s_2}(\Omega_2)}
 \lesssim
 \| f_1 \|_{L^2(\Omega_1)}
+
 \| f_2 \|_{L^2(\Omega_2)},
\label{eq:stability_estimate_continuous}
\end{equation}
with a hidden constant that is independent of $u$, $f_1$, and $f_2$.
\end{theorem}
\begin{proof}
 Set $v = u \in H$ in \eqref{eq:euler_lagrange_eq} to obtain $\mathcal{A}(u,u) = (f_1,u_1)_{L^2(\Omega_1)} + (f_2,u_2)_{L^2(\Omega_2)}$. The coercivity property of the bilinear form $\mathcal{A}$ stated in \eqref{eq:A_is_coercive}, combined with basic inequalities and the Poincar\'e inequality \eqref{eq:Poincare}, yields the desired stability estimate \eqref{eq:stability_estimate_continuous}. This concludes the proof.
\end{proof}

\subsection{Regularity properties}

We now examine regularity results for the solution of \eqref{eq:euler_lagrange_eq}.

\begin{theorem}[Sobolev regularity]\label{thm:Sob_reg}
Let $s_{1},s_{2} \in (0,1)$, let $f_{1} \in L^{2}(\Omega_{1})$, and let $f_{2} \in L^{2}(\Omega_{2})$. Let $u = u_{1} + u_{2} \in H$ be the unique solution to \eqref{eq:euler_lagrange_eq}. Then, for each $i \in \{1,2\}$, $u_{i} \in H^{s_{i} + \kappa_{i} - \varepsilon_{i}}(\Omega_{i})$ for all $0 < \varepsilon_{i} < s_{i}$, where $\kappa_{i} = \tfrac{1}{2}$ for $ \tfrac{1}{2} < s_{i} < 1$ and $\kappa_{i} = s_{i} - \varepsilon_{i}$ for $0 < s_{i} \leq \tfrac{1}{2}$. In addition, we have the bound
%\[
%\|u_{1}\|_{H^{s_{1} + \kappa_{1} - \varepsilon_{1}}(\Omega_{1})} \leq C_{1}\varepsilon_{1}^{-\nu_{1}}\|f_{1}\|_{L^{2}(\Omega_{1})}, \qquad \|u_{2}\|_{H^{s_{2} + \kappa_{2} - \varepsilon_{2}}(\Omega_{2})} \leq C_{2}\varepsilon_{2}^{-\nu_{2}}\|f_{2}\|_{L^{2}(\Omega_{2})},
%\]
\[
\|u_{i}\|_{H^{s_{i} + \kappa_{i} - \varepsilon_{i}}(\Omega_{i})} \leq C_{i}\varepsilon_{i}^{-\nu_{i}}
\left(\|f_{1}\|_{L^{2}(\Omega_{1})}
+
\|f_{2}\|_{L^{2}(\Omega_{2})}
\right), \qquad \forall \varepsilon_{i} \in (0,s_{i}),
\]
where $\nu_{i} = \tfrac{1}{2}$ for $\tfrac{1}{2} < s_{i} < 1$ and $\nu_{i} = \tfrac{1}{2} + \nu_{0,i}$ for $0 < s_{i} \leq \tfrac{1}{2}$. Here, $\nu_{0,i}$ is a positive constant depending on $\Omega_i$ and $n$, and $C_{i}$ is a positive constant depending on $\Omega_{1}$, $\Omega_2$, $n$, $s_{1}$, $s_2$, and $J$.
\end{theorem}
\begin{proof}
We provide a proof for $u_{1}$; the argument for $u_2$ is similar. From relation \eqref{eq:first_EL_eq}, and after some algebraic manipulations using the fact that $u_1$ and $v_1$ vanish outside $\Omega_1$, we conclude that $u_{1}$ is the weak solution of the following problem:
\begin{align}\label{eq:strong_prolem_u1}
u_1 \in \tilde{H}^{s_{1}}(\Omega_{1}):
\quad
(-\Delta)^{s_{1}}u_{1} = f_1 + 2\mathcal{C}_1 u_1 h - u_1 g + \vartheta =: \mathfrak{g}
% u_{1}(2\mathcal{C}_1 h - g) + \vartheta + f_{1}
\text{ in } \Omega_{1},
\end{align}
supplemented with the volumetric boundary condition $u_{1} = 0$ in $\Omega_{1}^{c}$. Here, the functions $h, g, \vartheta: \Omega_1 \rightarrow \mathbb{R}$ are defined as follows:
\begin{equation*}
h(x) = \int_{\Omega_{2}}|x-y|^{-n-2s_{1}}\mathrm{d}y, \quad g(x) = \int_{\Omega_{2}}J(x-y)\mathrm{d}y, \quad \vartheta(x) = \int_{\Omega_{2}}J(x-y)u_{2}(y) \mathrm{d}y.
\end{equation*}
Since $\mathrm{dist}(\Omega_1,\Omega_2) = d > 0$, it is immediate that the function $h$ belongs to $L^{\infty}(\Omega_1)$. On the other hand, since $J \in L^1(\mathbb{R}^n)$, the function $g$ belongs to  $L^{\infty}(\Omega_1)$. Since $h, g \in L^{\infty}(\Omega_1)$ and  $u_1 \in \tilde{H}^{s_1}(\Omega_1) \subset L^2(\Omega_1)$, we deduce that $u_{1}(2\mathcal{C}_1h - g)$ belongs to $L^{2}(\Omega_{1})$. Applying the continuity bound from Lemma \ref{lemma:cont_conv} shows that $\vartheta$ belongs to $L^{2}(\Omega_{1})$. Combining these results, we conclude that the forcing term $\mathfrak{g} \in L^{2}(\Omega_{1})$. Moreover, the stability bound \eqref{eq:stability_estimate_continuous}, the fact that $h,g \in L^{\infty}(\Omega_1)$, and Lemma \ref{lemma:cont_conv} yield $\|\mathfrak{g}\|_{L^2(\Omega_1)} \lesssim \|f_{1}\|_{L^{2}(\Omega_{1})} + \|f_{2}\|_{L^{2}(\Omega_{2})}$. The desired regularity then follows from a direct application of \cite[Theorem 2.1]{MR4283703}. This concludes the proof.
\end{proof}

\section{A finite element discretization}
\label{sec:fem}
Under the additional assumption that $\Omega_1$ and $\Omega_2$ are both Lipschitz polytopes, we introduce a finite element solution technique to approximate the solution of the coupled system \eqref{eq:first_EL_eq}--\eqref{eq:second_EL_eq}. We begin with some terminology and describe the construction of the underlying finite element spaces. For this purpose, we first introduce the families
\[
 \{ \T_{1,h} \}_{h>0},
 \qquad
 \{ \T_{2,\mathfrak{h}} \}_{\mathfrak{h}>0},
\]
of conforming and quasi-uniform meshes of $\bar{\Omega}_1$ and $\bar{\Omega}_2$, respectively, made of closed simplices $T$. Here, 
$
h:= \max \{ h_T : T \in \T_{1,h} \}
$
and
$
\mathfrak{h}:= \max \{ h_T : T \in \T_{2,\mathfrak{h}} \}
$
denote the mesh sizes of $\T_{1,h}$ and $\T_{2,\mathfrak{h}}$, respectively, and $h_T = \mathrm{diam}(T)$.

Given a mesh $\T_{1,h}$ and a mesh $\T_{2,\mathfrak{h}}$, we introduce the finite element spaces
\begin{align}
\mathbb{V}_{1,h} & := \{v_{h}\in C(\bar{\Omega}_1): v_{h}|_T \in \mathbb{P}_{1}(T) \ \forall T\in \T_{1,h}, v_h = 0 \text{ on } \partial\Omega_1\},
\\
\mathbb{V}_{2,\mathfrak{h}} & := \{v_{\mathfrak{h}}\in C(\bar{\Omega}_2): v_{\mathfrak{h}}|_T \in \mathbb{P}_{1}(T) \ \forall T\in \T_{2,\mathfrak{h}}, v_{\mathfrak{h}} = 0 \text{ on } \partial\Omega_2\}.
\end{align}

\begin{remark}[homogeneous Dirichlet boundary conditions]
 We note that we enforce a classical homogeneous Dirichlet boundary condition on $\partial \Omega_1$ ($\partial \Omega_2$) and that discrete functions in $\mathbb{V}_{1,h}$ ($\mathbb{V}_{2,\mathfrak{h}}$) can be trivially extended to $\Omega_1^c$ ($\Omega_2^c$) by zero.
 \label{rem:remark_Dirichlet}
\end{remark}

In view of the comments in Remark \ref{rem:remark_Dirichlet}, $\mathbb{V}_{1,h} \subset \tilde{H}^{s_1}(\Omega_1)$ and $\mathbb{V}_{2,\mathfrak{h}} \subset \tilde{H}^{s_2}(\Omega_2)$ for every $s_1,s_2 \in (0, 1)$.

\subsection{The discrete scheme}
We propose the following finite element approximation of the system \eqref{eq:first_EL_eq}--\eqref{eq:second_EL_eq}: Find $u_{1,h} \in \mathbb{V}_{1,h}$ and $u_{2,\mathfrak{h}} \in \mathbb{V}_{2,\mathfrak{h}}$ such that
\begin{align}\label{eq:first_EL_eq_discrete}
\begin{split}
& \mathcal{C}_{1}\int_{\Omega_{2}^{c}}\int_{\Omega_{2}^{c}}\dfrac{( u_{1,h}(x) - u_{1,h}(y))(v_{1,h}(x) - v_{1,h}(y))}{|x - y|^{n + 2s_{1}}}\mathrm{d}y\mathrm{d}x \\
& \quad + \int_{\Omega_{1}}\int_{\Omega_{2}}J(x - y)(u_{1,h}(x) - u_{2,\mathfrak{h}}(y))v_{1,h}(x)\mathrm{d}y\mathrm{d}x = \int_{\Omega_{1}}f_{1}(x)v_{1,h}(x) \mathrm{d}x
\end{split}
\end{align}
for all $v_{1,h} \in \mathbb{V}_{1,h}$, and
\begin{align}\label{eq:second_EL_eq_discrete}
\begin{split}
& \mathcal{C}_{2}\int_{\Omega_{1}^{c}}\int_{\Omega_{1}^{c}}\dfrac{(u_{2,\mathfrak{h}}(x) - u_{2,\mathfrak{h}}(y))(v_{2,\mathfrak{h}}(x) - v_{2,\mathfrak{h}}(y))}{|x - y|^{n + 2s_{2}}}\mathrm{d}y\mathrm{d}x \\
& \quad + \int_{\Omega_{1}}\int_{\Omega_{2}}J(x - y)(u_{2,\mathfrak{h}}(y) - u_{1,h}(x))v_{2,\mathfrak{h}}(y)\mathrm{d}y\mathrm{d}x = \int_{\Omega_{2}}f_{2}(x)v_{2,\mathfrak{h}}(x) \mathrm{d}x
\end{split}
\end{align}
for all $v_{2,\mathfrak{h}} \in \mathbb{V}_{2,\mathfrak{h}}$.

Define $\mathbb{H}_{h,\mathfrak{h}}:= \mathbb{V}_{1,h} \oplus \mathbb{V}_{2,\mathfrak{h}}$. If we add the discrete equations \eqref{eq:first_EL_eq_discrete} and \eqref{eq:second_EL_eq_discrete} and use the structure of the discrete space $\mathbb{H}_{h,\mathfrak{h}}$, we obtain the following equivalent weak formulation: Find $u_{h,\mathfrak{h}} = u_{1,h} + u_{2,\mathfrak{h}} \in \mathbb{H}_{h,\mathfrak{h}}$ such that
\begin{equation}
\quad
\mathcal{A}(u_{h,\mathfrak{h}}, v_{h,\mathfrak{h}}) = \int_{\Omega_{1}}f_{1}(x)v_{1,h}(x) \mathrm{d}x + \int_{\Omega_{2}}f_{2}(x)v_{2,\mathfrak{h}}(x) \mathrm{d}x
\label{eq:discrete_problem}
\end{equation}
for all  $v_{h,\mathfrak{h}} = v_{1,h} + v_{2,\mathfrak{h}}\in \mathbb{H}_{h,\mathfrak{h}}$.

We have the following result for the discrete problem.

\begin{theorem}[existence and uniqueness]
The discrete problem \eqref{eq:discrete_problem}, or equivalently the discrete coupled system \eqref{eq:first_EL_eq_discrete} and \eqref{eq:second_EL_eq_discrete}, admits a unique solution $u_{h,\mathfrak{h}} = u_{1,h} + u_{2,\mathfrak{h}} \in \mathbb{H}_{h,\mathfrak{h}} = \mathbb{V}_{1,h} \oplus \mathbb{V}_{2,\mathfrak{h}}$. In addition, we have the following stability bound:
\begin{equation}
|u_{1,h}|_{\tilde{H}^{s_1}(\Omega_1)}
 +
|u_{2,\mathfrak{h}}|_{\tilde{H}^{s_2}(\Omega_2)}
\lesssim
\| f_1 \|_{L^2(\Omega_1)}
+
\| f_2 \|_{L^2(\Omega_2)},
\label{eq:stability_estimate_discrete}
 \end{equation}
where the hidden constant is independent of $u_{1,h}$, $u_{2,\mathfrak{h}}$, $f_1$, $f_2$, and the discretization parameters $h$ and $\mathfrak{h}$.
\end{theorem}
\begin{proof}
Since $\mathcal{A}$ is bilinear, continuous, and coercive on $H \times H$ (see Lemma \ref{lemma:prop_A}), and since the finite--dimensional space $\mathbb{H}_{h,\mathfrak{h}} \subset H$, the existence and uniqueness of a solution $u_{h,\mathfrak{h}}$ follow directly from the Lax--Milgram lemma. The desired stability estimate \eqref{eq:stability_estimate_discrete} is obtained by substituting $v_{1,h} = u_{1,h}$ and $v_{2,\mathfrak{h}} = u_{2,\mathfrak{h}}$ into \eqref{eq:discrete_problem} and using the coercivity of the bilinear form $\mathcal{A}$ from \eqref{eq:A_is_coercive} as well as the Poincar\'e inequality \eqref{eq:Poincare}. This concludes the proof.
\end{proof}

\subsection{A priori error bounds}\label{sec:error}

We begin this subsection with the following result.

\begin{lemma}[Galerkin orthogonality]
Let $u \in H$ and let $u_{h,\mathfrak{h}} \in \mathbb{H}_{h,\mathfrak{h}}$ be the solutions of problems \eqref{eq:euler_lagrange_eq} and \eqref{eq:discrete_problem}, respectively. Then, 
\begin{equation}\label{eq:Galerkin_orthog}
\mathcal{A}(u - u_{h,\mathfrak{h}}, v_{h,\mathfrak{h}}) = 0 \qquad \forall v_{h,\mathfrak{h}} \in \mathbb{H}_{h,\mathfrak{h}}.
\end{equation}
\end{lemma}
\begin{proof}
Since $\mathbb{H}_{h,\mathfrak{h}} \subset H$, we are allowed to set $v_{h,\mathfrak{h}} = v_{1,h} + v_{2,\mathfrak{h}} \in \mathbb{H}_{h,\mathfrak{h}}$ as a test function in problem \eqref{eq:euler_lagrange_eq}. We obtain
\[
 \mathcal{A}(u,v_{h,\mathfrak{h}}) = \int_{\Omega_{1}}f_{1}(x)v_{1,h}(x) \mathrm{d}x + \int_{\Omega_{2}}f_{2}(x)v_{2,\mathfrak{h}}(x) \mathrm{d}x \quad \forall v_{h,\mathfrak{h}} \in \mathbb{H}_{h,\mathfrak{h}}.
\]
If we subtract the discrete equation \eqref{eq:discrete_problem} from the previous relation, we obtain the desired Galerkin orthogonality property.
\end{proof}

We now derive the following quasi-best approximation result.

\begin{lemma}[Cea's lemma]
Let $u = u_{1} + u_{2} \in H$ and let $u_{h,\mathfrak{h}} = u_{1,h} + u_{2,\mathfrak{h}} \in \mathbb{H}_{h,\mathfrak{h}}$ be the solutions of problems \eqref{eq:euler_lagrange_eq} and \eqref{eq:discrete_problem}, respectively. Then,
%\begin{equation}\label{eq:Ceas_lemma}
%|u - u_{h,\mathfrak{h}}|_{H} \lesssim \min_{v_{h,\mathfrak{h}} \in \mathbb{H}_{h,\mathfrak{h}}} |u - v_{h,\mathfrak{h}}|_{H}.
%\end{equation}
\begin{align}\label{eq:Ceas_lemma}
\begin{split}
& |u_{1} - u_{1,h}|^{2}_{\tilde{H}^{s_{1}}(\Omega_{1})} + |u_{2} - u_{2,\mathfrak{h}}|_{\tilde{H}^{s_{2}}(\Omega_{2})}^{2} 
\\
& \lesssim \min_{v_{h,\mathfrak{h}} \in \mathbb{H}_{h,\mathfrak{h}}}
\left( 
|u_{1} - v_{1,h}|_{\tilde{H}^{s_{1}}(\Omega_{1})}^{2} + |u_{2} - v_{2,\mathfrak{h}}|_{\tilde{H}^{s_{2}}(\Omega_{2})}^{2}
\right).
\end{split}
\end{align}
\end{lemma}
\begin{proof}
The proof is standard. We provide a brief proof for completeness. Using the coercivity and continuity of the bilinear form $\mathcal{A}$ established in Lemma \ref{lemma:prop_A} and the Galerkin orthogonality property \eqref{eq:Galerkin_orthog}, it follows that
\begin{align*}
\alpha|u - u_{h,\mathfrak{h}}|_{H}^{2} & \leq \mathcal{A}(u - u_{h,\mathfrak{h}},u - u_{h,\mathfrak{h}}) \\
& = \mathcal{A}(u - u_{h,\mathfrak{h}},u - v_{h,\mathfrak{h}}) \leq M |u - u_{h,\mathfrak{h}}|_{H}|u - v_{h,\mathfrak{h}}|_{H} \qquad \forall v_{h,\mathfrak{h}} \in \mathbb{H}_{h,\mathfrak{h}}.
\end{align*}
Taking the minimum over $v_{h,\mathfrak{h}} \in \mathbb{H}_{h,\mathfrak{h}}$ yields the desired estimate \eqref{eq:Ceas_lemma}.
\end{proof}

We are now in a position to state and prove the following a priori error bound.

\begin{theorem}[a priori error bound]\label{thm:apriori_error}
Let $s_{1},s_{2} \in (0,1)$, let $f_{1} \in L^{2}(\Omega_{1})$, and let $f_{2} \in L^{2}(\Omega_{2})$. Let $u = u_{1} + u_{2} \in H$ and let $u_{h,\mathfrak{h}} = u_{1,h} + u_{2,\mathfrak{h}} \in \mathbb{H}_{h,\mathfrak{h}}$ be the solutions of problems \eqref{eq:euler_lagrange_eq} and \eqref{eq:discrete_problem}, respectively. Then, we have the following error bound:
\begin{multline}\label{eq:apriori_est}
|u_{1} - u_{1,h}|_{\tilde{H}^{s_{1}}(\Omega_{1})}^{2} + |u_{2} - u_{2,\mathfrak{h}}|_{\tilde{H}^{s_{2}}(\Omega_{2})}^{2}
\\
\lesssim 
\left( h^{2\gamma_{1}}|\log h|^{2\varphi_{1}} + \mathfrak{h}^{2\gamma_{2}}|\log \mathfrak{h}|^{2\varphi_{2}} \right)
\left( \|f_{1}\|_{L^{2}(\Omega_{1})}^{2}
+
\|f_{2}\|_{L^{2}(\Omega_{2})}^{2} \right).
\end{multline}
Here, for $i \in \{1,2\}$, $\gamma_{i} = \min\{s_{i},\tfrac{1}{2}\}$, $\varphi_{i} = \nu_{i}$ if $s_{i} \neq \tfrac{1}{2}$, $\varphi_{i} = 1 + \nu_{i}$ if $s_{i} = \tfrac{1}{2}$, and $\nu_{i} \geq \tfrac{1}{2}$ is the constant in Theorem \ref{thm:Sob_reg}.
\end{theorem}
\begin{proof}
 With the quasi-best approximation estimate \eqref{eq:Ceas_lemma} in hand, we can consider $v_{h,\mathfrak{h}} = \Pi_h u_1 + \Pi_{\mathfrak{h}} u_2$, where $\Pi_h$ and $\Pi_{\mathfrak{h}}$ denote appropriate quasi-interpolation operators (possible choices include the Scott--Zhang and Cl\'ement interpolation operators), and obtain the bound
 \[
|u_{1} - u_{1,h}|^{2}_{\tilde{H}^{s_{1}}(\Omega_{1})} + |u_{2} - u_{2,\mathfrak{h}}|_{\tilde{H}^{s_{2}}(\Omega_{2})}^{2} 
\lesssim 
|u_{1} - \Pi_h u_1 |_{\tilde{H}^{s_{1}}(\Omega_{1})}^{2} + |u_{2} - \Pi_{\mathfrak{h}} u_2|_{\tilde{H}^{s_{2}}(\Omega_{2})}^{2}.
 \]
The desired estimate follows from the arguments given in the proof of \cite[Theorem 3.5]{MR4283703}, combined with the regularity results of Theorem \ref{thm:Sob_reg}. This concludes the proof.
\end{proof}

\section{Alternating schemes}\label{sec:altern_schemes}
Inspired by the method proposed by Schwarz in \cite{schwarz1870ueber}, we propose and analyze the continuous alternating \textbf{Algorithm} \ref{Algorithm_cont} in this section and prove that it converges to the solution of the coupled problem \eqref{eq:first_EL_eq}--\eqref{eq:second_EL_eq}. Although it is already established that the coupled problem \eqref{eq:first_EL_eq}--\eqref{eq:second_EL_eq} has a unique solution (see Theorem \ref{thm:existence_of_minimizer}), the analysis of this algorithm provides the basis for proposing and analyzing the fully discrete alternating \textbf{Algorithm} \ref{Algorithm_disc}, an iterative algorithm that converges to the solution of the finite element approximation \eqref{eq:first_EL_eq_discrete}--\eqref{eq:second_EL_eq_discrete}.

\subsection{The continuous alternating scheme}
\label{sec:continuous_alternating_scheme}

We begin this section by introducing the continuous alternating scheme mentioned above (\textbf{Algorithm} \ref{Algorithm_cont}).

%%%%%%%%%%%%%%%%%%%%%%%%%%%%%%%%%%%%%%%%%%%%%%%%%%%%%%%
%%%%%%%%%%%%%%%%%  AlGORITHM 1    %%%%%%%%%%%%%%%%%%%%%
%%%%%%%%%%%%%%%%%%%%%%%%%%%%%%%%%%%%%%%%%%%%%%%%%%%%%%%
\begin{algorithm}[ht]
\caption{\textbf{The continuous alternating scheme}}
\label{Algorithm_cont}
\textbf{Input:} $\Omega_1, \Omega_2 \subset \mathbb{R}^n$, $s_{1},s_{2} \in (0,1)$, $f_{1} \in L^{2}(\Omega_{1})$, $f_{2} \in L^{2}(\Omega_{2})$, $J \in L^{1}(\mathbb{R}^{n})$, and $u_{2}^{0} \in \tilde{H}^{s_2} (\Omega_2)$.
\\
%\EO{\textbf{Step $0$:} }
\textbf{Step $\boldsymbol{0}$}: Define $u^0 = 0 + u_2^{0} \in H$.
\\
\textbf{For} $i=1$ \textbf{until convergence do}
\begin{itemize}
\item[$\boldsymbol{1}$:] Find the solution $u_{1}^{2i-1} \in \tilde{H}^{s_1}(\Omega_1)$ of \eqref{eq:first_EL_eq} with $u_{2}$ replaced by $u_{2}^{2i-2}$.
\item[$\boldsymbol{2}$:] Define $u_2^{2i-1} = u_2^{2i-2} \in \tilde{H}^{s_2}(\Omega_2)$ and $u^{2i-1} = u_{1}^{2i-1} + u_{2}^{2i-1} \in H$.
\item[$\boldsymbol{3}$:] Find the solution $u_{2}^{2i} \in \tilde{H}^{s_2}(\Omega_2)$ of \eqref{eq:second_EL_eq} with $u_{1}$ replaced by $u_{1}^{2i-1}$.
\item[$\boldsymbol{4}$:] Define $u_1^{2i} = u_1^{2i-1} \in \tilde{H}^{s_1}(\Omega_1)$ and $u^{2i} = u_{1}^{2i} + u_{2}^{2i} \in H$.
\end{itemize}
\end{algorithm}

\subsubsection{Solution operators}
To analyze \textbf{Algorithm} \ref{Algorithm_cont}, we introduce suitable solution operators. Given forcing terms $f_{1} \in L^{2}(\Omega_{1})$ and $f_{2} \in L^{2}(\Omega_{2})$, we define
\begin{align}
\label{eq:L_1}
\mathcal{L}_{f_{1}}: L^{2}(\Omega_{2}) \rightarrow \tilde{H}^{s_{1}}(\Omega_{1}), 
\qquad 
w_2 \mapsto \mathfrak{u}_{1} = \mathcal{L}_{f_{1}}(w_{2}),
\\
\label{eq:L_2}
\mathcal{L}_{f_{2}}: L^{2}(\Omega_{1}) \rightarrow \tilde{H}^{s_{2}}(\Omega_{2}), 
\qquad 
w_1 \mapsto \mathfrak{u}_{2} = \mathcal{L}_{f_{2}}(w_{1}).
\end{align}
Here, $\mathfrak{u}_{1}$ corresponds to the solution of \eqref{eq:first_EL_eq}, where $u_2$ is replaced by $w_{2}$, and $\mathfrak{u}_{2}$ corresponds to the solution of \eqref{eq:second_EL_eq}, where $u_1$ is replaced by $w_{1}$.
% Let $i \in \{1,2\}$. In the case where $f_{i} \equiv 0$, we simply write $\mathcal{L}_i$. We note that $\mathcal{L}_i$ is a linear operator.

We now prove that the operators $\mathcal{L}_{f_{1}}$ and $\mathcal{L}_{f_{2}}$ are well-defined and continuous.

\begin{lemma}[$\mathcal{L}_{f_{1}}$ and $\mathcal{L}_{f_{2}}$]
The operators $\mathcal{L}_{f_{1}}$ and $\mathcal{L}_{f_{2}}$, defined in \eqref{eq:L_1} and \eqref{eq:L_2}, respectively, are well-defined and continuous, and satisfy the following bounds:
\begin{align}
\label{eq:Lf1_bounded}
 | \mathcal{L}_{f_{1}}(w_2) |_{\tilde{H}^{s_1}(\Omega_1)}
 & \lesssim 
 \| f_1 \|_{L^2(\Omega_1)} + \| J \|_{L^1(\mathbb{R}^n)} \| w_2 \|_{L^2(\Omega_2)},
\\
\label{eq:Lf2_bounded}
| \mathcal{L}_{f_{2}}(w_1) |_{\tilde{H}^{s_2}(\Omega_2)}
& \lesssim 
\| f_2 \|_{L^2(\Omega_2)} + \| J \|_{L^1(\mathbb{R}^n)} \| w_1 \|_{L^2(\Omega_1)}.
\end{align}
\end{lemma}
\begin{proof}
We analyze the operator $\mathcal{L}_{f_{1}}$; the analysis for the operator $\mathcal{L}_{f_{2}}$ is similar.

Let $w_2 \in L^2(\Omega_2)$. Define $\mathfrak{E}_1: \tilde{H}^{s_1}(\Omega_1) \rightarrow \mathbb{R}^{+}_0$ and $\mathcal{E}_1: \tilde{H}^{s_1}(\Omega_1) \rightarrow \mathbb{R}$ as follows:
\begin{align}
\mathfrak{E}_1(\mathfrak{u}_1) &:=
\frac{\mathcal{C}_{1}}{2} \int_{\Omega_{2}^c}\int_{\Omega_{2}^c}\frac{(\mathfrak{u}_1(x) - \mathfrak{u}_1(y))^2}{|x - y|^{n + 2s_{1}}}\mathrm{d}y\mathrm{d}x
+
\dfrac{1}{2}\int_{\Omega_1}\int_{\Omega_{2}}J(x - y)\mathfrak{u}_1^2(x) \mathrm{d}y\mathrm{d}x,
\\
\mathcal{E}_1(\mathfrak{u}_1) &:=
\mathfrak{E}_1(\mathfrak{u}_1) - \int_{\Omega_1} f_1(x) \mathfrak{u}_1(x)\mathrm{d}x
-
\int_{\Omega_1}\int_{\Omega_{2}}J(x - y) w_2(y)\mathfrak{u}_1(x) \mathrm{d}y\mathrm{d}x.
\end{align}
We first note that $\mathfrak{E}_1(\mathfrak{u}_1) \geq \mathtt{C}_1 \| \mathfrak{u}_1 \|^2_{L^2(\Omega_1)}$ for all $\mathfrak{u}_1 \in \tilde{H}^{s_1}(\Omega_1)$, where $\mathtt{C}_1$ is a positive constant. This result follows from an adaptation of the arguments in the proof of Lemma \ref{lemma:L2_lower_bound_E}. With this bound, we follow the arguments in the proof of Theorem \ref{thm:existence_of_minimizer} and use \eqref{eq:first_ineq_tildeHs_Hs_norm}, the nonnegativity property of the kernel $J$, and the continuity property of Lemma \ref{lemma:cont_conv} to derive the following coercivity properties:
\begin{align}
\label{eq:Lf1_coercive_1}
\mathfrak{E}_1 (\mathfrak{u}_1) & \gtrsim | \mathfrak{u}_1 |^2_{\tilde{H}^{s_1}(\Omega_1)}
\qquad
\forall \mathfrak{u}_1 \in \tilde{H}^{s_1}(\Omega_1),
\\
\label{eq:Lf1_coercive_2}
\mathcal{E}_1 (\mathfrak{u}_1) & \gtrsim | \mathfrak{u}_1 |^2_{\tilde{H}^{s_1}(\Omega_1)}- \|f_{1}\|^{2}_{L^{2}(\Omega_{1})} - \|J\|^{2}_{L^{1}(\mathbb{R}^{n})} \| w_2 \|_{L^{2}(\Omega_{2})}^{2}
\qquad
\forall \mathfrak{u}_1 \in \tilde{H}^{s_1}(\Omega_1).
\end{align}
Let us also note that $\mathcal{E}_1$ is convex and continuous in $\tilde{H}^{s_1}(\Omega_1)$. The convexity of $\mathcal{E}_1$ is clear, and its continuity follows from the arguments presented in the proof of Lemma \ref{lemma:mathcalE_cont_convex}. With these properties, the direct method of the calculus of variations allows us to deduce the existence of a minimizer $\mathfrak{u}_1$ of $\mathcal{E}_1$ in $\tilde{H}^{s_{1}}(\Omega_{1})$. The strict convexity of $\mathcal{E}_1$ guarantees the uniqueness of $\mathfrak{u}_1$. Finally, as in section \ref{sec:euler_lagrange}, it can be shown that $\mathfrak{u}_1$ is the unique solution to the problem
\begin{multline}
\mathcal{C}_{1} \int_{\Omega_{2}^c} \int_{\Omega_{2}^c}\frac{(\mathfrak{u}_1(x) - \mathfrak{u}_1(y))(v(x) - v(y))}{|x - y|^{n + 2s_{1}}}\mathrm{d}y\mathrm{d}x
\\
+
\int_{\Omega_1}\int_{\Omega_{2}}J(x - y)\mathfrak{u}_1(x) v(x) \mathrm{d}y\mathrm{d}x 
= 
\int_{\Omega_1} f_1(x) v(x)\mathrm{d}x
+ 
\int_{\Omega_{1}}\int_{\Omega_{2}}J(x-y)w_2(y)v(x) \mathrm{d}y\mathrm{d}x
\label{eq:Lf1_equation}
\end{multline}
for all $v \in \tilde{H}^{s_1}(\Omega_1)$, i.e., $\mathcal{L}_{f_{1}}(w_2) = \mathfrak{u}_1$. This shows that $\mathcal{L}_{f_{1}}$ is well-defined. We now set $v = \mathfrak{u}_1$ in \eqref{eq:Lf1_equation} and use the nonnegativity of $J$, the Cauchy--Schwarz inequality, the Poincaré inequality \eqref{eq:Poincare}, Lemma \ref{lemma:cont_conv}, and Young's inequality to obtain a stability bound for $\mathfrak{u}_1$, which implies the desired bound for $\mathcal{L}_{f_{1}}$.
\end{proof}

Let $i \in \{1,2\}$. In the case where $f_{i} \equiv 0$, we simply write $\mathcal{L}_i$. Since problem \eqref{eq:Lf1_equation} is linear in both $\mathfrak{u}_1$ and $w_2$, and its solution is unique (as established in the proof above), it follows that $\mathcal{L}_1$ is a linear operator; the same result holds for $\mathcal{L}_2$.

Having defined the operators $\mathcal{L}_{f_{1}}$ and $\mathcal{L}_{f_{2}}$, we rewrite the continuous alternating scheme as follows: Given $u_{2}^0 \in \tilde{H}^{s_2}(\Omega_{2})$, we define $u^0 = 0 + u_2^0 \in H$, where $0$ denotes the zero function on $\Omega_1$, and compute, for $i=1$ until convergence,
\begin{equation}
\label{eq:schwartz_method}
\begin{aligned}
& u_{1}^{2i-1}  = \mathcal{L}_{f_{1}}(u_{2}^{2i-2}),
\qquad
u_2^{2i-1} = u_2^{2i-2},
\qquad
u^{2i-1} = u_1^{2i-1} + u_2^{2i-1} \in H,
\\
& u_{2}^{2i}  = \mathcal{L}_{f_{2}}(u_{1}^{2i-1}),
\qquad
u_{1}^{2i} = u_{1}^{2i-1},
\qquad
u^{2i} = u_1^{2i} + u_2^{2i} \in H.
\end{aligned}
\end{equation}

\subsubsection{Definition of convergence and equivalence}
We say that \textbf{Algorithm} \ref{Algorithm_cont} converges if the sequence $\{ u^{i} \}_{i \geq 0} \subset H$, where $u^{i} = u^i_{1} + u^i_{2}$, converges to the solution $u = u_{1} + u_{2}$ of problem \eqref{eq:first_EL_eq}--\eqref{eq:second_EL_eq} in the following sense:
\[
u^{i}  \rightarrow u ~\mathrm{in}~H,
\qquad
i \uparrow \infty.
\]

\begin{remark}[equivalence]
\label{rem:conv_schwartz_method}
Let $f_1 \in L^2(\Omega_1)$, $f_2 \in L^2(\Omega_2)$, and let $u = u_1 + u_2$ be the unique solution of the coupled system \eqref{eq:first_EL_eq}--\eqref{eq:second_EL_eq}. The key observation is that the error sequence $\{ u - u^{i} \}_{i \geq 0}$ satisfies the same iterative scheme as \eqref{eq:schwartz_method} but with
$f_1 \equiv f_2 \equiv 0$ and initial datum $u_2 - u_2^0$. Indeed, using
the affine dependence of $\mathcal{L}_{f_1}$ and $\mathcal{L}_{f_2}$ on
their arguments, the sequence $\{u - u^i\}_{i\geq 0}$ verifies for every
$i \geq 1$:
\begin{equation}
\label{eq:schwartz_method_0_0_equivalence}
\begin{aligned}
& u_1 - u_{1}^{2i-1}  = \mathcal{L}_{f_1}(u_2) - \mathcal{L}_{f_1}(u_2^{2i-2}) = \mathcal{L}_{1}(u_2 - u_{2}^{2i-2}),
\\
& u_2 - u_{2}^{2i-1}  = u_2 - u_2^{2i-2},
\\
& u_2 - u_{2}^{2i}  = \mathcal{L}_{f_2}(u_1) - \mathcal{L}_{f_2}(u_1^{2i-1}) = \mathcal{L}_{2}(u_1 - u_{1}^{2i-1}),
\\
& u_1 - u_{1}^{2i}  = u_1 - u_{1}^{2i-1},
\end{aligned}
\end{equation}
where $\mathcal{L}_i$ is defined as in \eqref{eq:L_1}--\eqref{eq:L_2} with $f_i \equiv 0$. Therefore, it suffices to prove convergence of \eqref{eq:schwartz_method} in the case $f_1 \equiv f_2 \equiv 0$, which corresponds to the following iterative method: Given $u_2^0 \in \tilde{H}^{s_2}(\Omega_2)$, define $\mathtt{u}^{0} = 0 + u_2^{0} \in H$, where $0$ denotes the zero function on $\Omega_1$, and compute, for $i=1$ until convergence,
\begin{equation}
\label{eq:schwartz_method_0_0}
\begin{aligned}
& \mathtt{u}_{1}^{2i-1} = \mathcal{L}_{1}(\mathtt{u}_{2}^{2i-2}),
\qquad
\mathtt{u}_{2}^{2i-1} = \mathtt{u}_{2}^{2i-2},
\qquad
\mathtt{u}^{2i-1} = \mathtt{u}_{1}^{2i-1} + \mathtt{u}_{2}^{2i-1} \in H,
\\
& \mathtt{u}_{2}^{2i} = \mathcal{L}_{2}(\mathtt{u}_{1}^{2i-1}),
\qquad
\mathtt{u}_{1}^{2i} = \mathtt{u}_{1}^{2i-1},
\qquad
\mathtt{u}^{2i} = \mathtt{u}_{1}^{2i} + \mathtt{u}_{2}^{2i} \in H.
\end{aligned}
\end{equation}
Here, $\{ \mathtt{u}^i \}_{i \geq 0}$ denotes the sequence corresponding to the case $f_1 \equiv f_2 \equiv 0$, to distinguish it from the original sequence $\{ u^i \}_{i\geq 0}$.
\end{remark}

\subsubsection{Convergence analysis}
We define 
\begin{align}
V_{1} = \{u = u_{1} + u_{2} \in H: u_{1} = \mathcal{L}_{1}(u_{2}),\, u_{2} \in \tilde{H}^{s_2}(\Omega_{2}) \}, \\
V_{2} = \{u = u_{1} + u_{2} \in H: u_{2} = \mathcal{L}_{2}(u_{1}),\, u_{1} \in \tilde{H}^{s_1}(\Omega_{1}) \}.
\end{align}
Since $\mathcal{L}_{1}$ and $\mathcal{L}_{2}$ are both linear, we deduce that $V_{1}$ and $V_{2}$ are both subspaces of $H$.

\begin{lemma}[$H = V_{1} \oplus V_{2}$]
It holds that $H = V_{1} \oplus V_{2}$.
\label{lemma:V1+V2}
\end{lemma}
\begin{proof} 
We divide the proof into two steps.

\emph{Step 1.} $V_{1} \cap V_{2} = \{0\}$. Let $u = u_{1} + u_{2} \in V_1 \cap V_{2}$. From the definition of $V_{1}$ and $V_{2}$ we deduce that $u_{1} = \mathcal{L}_{1}(u_{2}) \in \tilde{H}^{s_1}(\Omega_1)$ and $u_{2} = \mathcal{L}_{2}(u_1) \in \tilde{H}^{s_2}(\Omega_2)$. We now use the definition of the operators $\mathcal{L}_{1}$ and $\mathcal{L}_{2}$ from \eqref{eq:L_1} and \eqref{eq:L_2}, respectively (with $f_1 \equiv f_2 \equiv 0$), to conclude that $u = u_1 + u_2 \in H$ solves the coupled problem \eqref{eq:first_EL_eq}--\eqref{eq:second_EL_eq} with $f_1 \equiv f_2 \equiv 0$. Since this problem is well-posed, we immediately deduce that $u_1 \equiv u_2 \equiv 0$ and then that $u \equiv 0$.

\emph{Step 2.} $H = V_{1} + V_{2}$. Let $u = u_1 + u_2 \in H$. We need to prove that there exist $v \in V_1$ and $w \in V_2$ such that $u = v + w$. By the definition of the subspaces $V_1$ and $V_2$, we have that $v = v_1 + v_2$, where $v_1 = \mathcal{L}_1(v_2)$, and that $w = w_1 + w_2$, where $w_2 = \mathcal{L}_2(w_1)$. In view of this fact, in the following we prove the existence of $v_2$ and $w_1$, so that $u_1 = \mathcal{L}_1(v_2) + w_1$ and $u_2 = v_2 + \mathcal{L}_2(w_1)$.

\emph{Step 2.1} \emph{Existence of $v_2$.} Applying the linear operator $\mathcal{L}_2$ to the relation $u_1 = \mathcal{L}_1(v_2) + w_1$ gives $\mathcal{L}_2(u_1) = \mathcal{L}_2(\mathcal{L}_1(v_2)) + \mathcal{L}_2(w_1)$. We use this relation and $u_2 = v_2 + \mathcal{L}_2(w_1)$ to obtain $\mathcal{L}_2(u_1) = \mathcal{L}_2(\mathcal{L}_1(v_2)) + u_2 - v_2$, which can be rewritten as follows:
\begin{equation}
 u_2 - \mathcal{L}_2(u_1) = v_2 - \mathcal{L}_2(\mathcal{L}_1(v_2)) = (I - \mathcal{L}_2 \circ \mathcal{L}_1) v_2.
 \label{eq:v_2}
\end{equation}
We now note that the operator $\mathcal{L}_2 \circ \mathcal{L}_1: L^2(\Omega_2) \rightarrow \tilde{H}^{s_2}(\Omega_{2})$. Given the compact embedding $\tilde{H}^{s_2}(\Omega_2) \hookrightarrow L^2(\Omega_2)$, we can consider $\mathcal{L}_2 \circ \mathcal{L}_1: L^2(\Omega_2) \rightarrow L^2(\Omega_2)$, which is thus a linear and compact operator. 

With this setting in hand, let us now show that there exists a unique $v_2$ that verifies \eqref{eq:v_2}. To accomplish this task, we will rely on the Fredholm alternative and prove that $\mathrm{Ker}(I - \mathcal{L}_2 \circ \mathcal{L}_1) = \{ 0 \}$ to conclude that $(I - \mathcal{L}_2 \circ \mathcal{L}_1)$ is a bijection.  Let $\mathfrak{v}_2 \in L^2(\Omega_2)$ be such that $(I - \mathcal{L}_2 \circ \mathcal{L}_1) \mathfrak{v}_2 = 0$. This implies that $(\mathcal{L}_2 \circ \mathcal{L}_1) (\mathfrak{v}_2) = \mathfrak{v}_2$. Using the definition of the linear maps $\mathcal{L}_1$ and $\mathcal{L}_2$, we can show that $\mathcal{L}_1(\mathfrak{v}_2) + \mathfrak{v}_2 = \mathcal{L}_1(\mathfrak{v}_2) + \mathcal{L}_2( \mathcal{L}_1 (\mathfrak{v}_2)) \in H$ solves the coupled problem \eqref{eq:first_EL_eq}--\eqref{eq:second_EL_eq} with $f_1 \equiv f_2 \equiv 0$. Since this problem is well-posed, we can thus conclude that $\mathfrak{v}_2 = 0$. It follows that $I - \mathcal{L}_2 \circ \mathcal{L}_1 : L^2(\Omega_2) \rightarrow L^2(\Omega_2)$ is a bijection and thus that there exists a unique $v_2$ that verifies \eqref{eq:v_2}.

%\emph{Step 2.2} \emph{Existence of $w_1$.} Similar arguments to the ones elaborated in Step 2.1 allows us to conclude the existence of a unique $w_1$.

\emph{Step 2.2} \emph{Existence of $w_1$.} Having determined $v_{2}$, we consider $w_{1} \in \tilde{H}^{s_1}(\Omega_1)$ as
\begin{equation}\label{eq:w_1}
w_{1} = u_{1} - \mathcal{L}_{1}(v_{2}).
\end{equation}

\emph{Step 2.3.} $u = v + w$. From the relations \eqref{eq:v_2} and \eqref{eq:w_1} we obtain
\begin{align*}
u_{1} = w_{1} + \mathcal{L}_{1}(v_{2}),
\qquad
u_{2}
% = v_{2} + \mathcal{L}_{2}(u_{1} - \mathcal{L}_{1}(v_{2}))
= v_{2} + \mathcal{L}_{2}(w_{1}),
\end{align*}
as we intended to show. This concludes the proof.
\end{proof} 

To continue our analysis, we define the linear operators
\begin{align}
\label{eq:P_1}
\mathcal{P}_{1}: H \rightarrow H, 
\qquad 
u = u_{1} + u_{2} \mapsto \mathcal{P}_{1}(u) = \mathcal{L}_{1}(u_{2}) + u_{2},
\\
\label{eq:P_2}
\mathcal{P}_{2}: H \rightarrow H, 
\qquad 
u = u_{1} + u_{2} \mapsto \mathcal{P}_{2}(u) = u_{1} + \mathcal{L}_{2}(u_{1}).
\end{align}
We note that $\mathcal{P}_{1}$ and $\mathcal{P}_{2}$ are linear. In the following, we analyze orthogonality properties of these operators with respect to the bilinear form $\mathcal{A}$, which induces an inner product in $H \times H$; see Remark \ref{rem:inner_product}. For this purpose, we define
\begin{equation}
 \langle u,v \rangle_{\mathcal{A}} := \mathcal{A}(u,v),
 \qquad
 |v|_{\mathcal{A}} := \langle v,v \rangle_{\mathcal{A}}^{\frac{1}{2}}
 \qquad
 \forall u,v \in H.
 \label{eq:inner_product_A}
\end{equation}

\begin{lemma}[orthogonal projection]\label{lemma:orth_proj}
Let $i\in \{1,2\}$. The map $\mathcal{P}_{i}$ is an orthogonal projection from $H$ onto $V_{i}$ with respect to the inner product defined in \eqref{eq:inner_product_A}.
\end{lemma}
\begin{proof}
We analyze the operator $\mathcal{P}_{1}: H \rightarrow V_1$; the analysis for $\mathcal{P}_{2}$ is similar.

The fact that $\mathcal{P}_1$ is a projection is trivial. In fact, $\mathcal{P}_1(\mathcal{P}_1(u)) = \mathcal{P}_1(\mathcal{L}_{1}(u_{2}) + u_{2}) = \mathcal{L}_{1}(u_{2}) + u_2 = \mathcal{P}_{1}(u)$ for all $u = u_1 + u_2 \in H$. We now prove that $\mathcal{P}_1: H \rightarrow V_1$ is an orthogonal map with respect to the inner product 
$\langle \cdot, \cdot \rangle_{\mathcal{A}}$ defined in \eqref{eq:inner_product_A}. Given $u = u_{1} + u_{2} \in H$, we prove below that
\begin{equation}\label{eq:orthog_prop_P1}
\langle u - \mathcal{P}_{1}(u),v \rangle_{\mathcal{A}} = \langle (u_{1} - \mathcal{L}_{1}(u_{2})) + 0,v \rangle_{\mathcal{A}} = 0 \qquad \forall v = \mathcal{L}_{1}(v_{2}) + v_{2} \in V_{1}.
\end{equation}
Denote $U_{1} = u_{1} - \mathcal{L}_{1}(u_{2}) \in \tilde{H}^{s_1}(\Omega_1)$. From the definition of  the inner product $\langle \cdot , \cdot \rangle_{\mathcal{A}}$ and the bilinear form $\mathcal{A}(\cdot,\cdot)$, it follows that
\begin{align*}
\begin{split}
%\langle u - \mathcal{P}_{1}(u),v \rangle_{H}
\langle u - \mathcal{P}_{1}(u),v \rangle_{\mathcal{A}} & = \mathcal{C}_{1}\int_{\Omega_{2}^{c}}\int_{\Omega_{2}^{c}}\dfrac{(U_{1}(x) - U_{1}(y))(\mathcal{L}_{1}(v_{2})(x) - \mathcal{L}_{1}(v_{2})(y))}{|x - y|^{n + 2s_{1}}}\mathrm{d}y\mathrm{d}x 
\\ 
& + \int_{\Omega_{1}}\int_{\Omega_{2}}J(x - y)U_{1}(x)(\mathcal{L}_{1}(v_{2})(x) - v_{2}(y)) \mathrm{d}y\mathrm{d}x.
\end{split}
\end{align*}
On the other hand, let us note that $\mathcal{L}_{1}(v_{2}) \in \tilde{H}^{s_{1}}(\Omega_{1})$ solves the following problem:
\begin{align*}
& \mathcal{C}_{1}\int_{\Omega_{2}^{c}}\int_{\Omega_{2}^{c}}\dfrac{(\mathcal{L}_{1}(v_{2})(x) - \mathcal{L}_{1}(v_{2})(y))(z_{1}(x) - z_{1}(y))}{|x - y|^{n + 2s_{1}}}\mathrm{d}y\mathrm{d}x 
\\ 
&
+
\int_{\Omega_{1}}\int_{\Omega_{2}}J(x - y)((\mathcal{L}_{1}(v_{2})(x) - v_{2}(y))z_{1}(x) \mathrm{d}y\mathrm{d}x = 0 \qquad \forall z_{1} \in \tilde{H}^{s_{1}}(\Omega_{1}).
\end{align*} 
If we set $z_{1} = U_{1}$ in the previous weak formulation, we obtain \eqref{eq:orthog_prop_P1}, as we intended to show. This shows that $\mathcal{P}_1: H \rightarrow V_1$ is an orthogonal map with respect to 
%$\langle \cdot, \cdot \rangle_H$.
$\langle \cdot, \cdot \rangle_{\mathcal{A}}$.
\end{proof}

To present the main result of this section, we note that the elements of the sequence $\{ \mathtt{u}^i \}_{i \geq 0}$ defined in \eqref{eq:schwartz_method_0_0} can be rewritten in terms of the orthogonal projections $\mathcal{P}_1$ and $\mathcal{P}_2$, defined in \eqref{eq:P_1} and \eqref{eq:P_2}, respectively, as follows: $\mathtt{u}^{0} = 0 + u_{2}^{0}$ and for every $i \geq 1$,
\begin{align}
& \mathtt{u}^{2i-1}  = \mathcal{L}_1(\mathtt{u}_2^{2i-2}) + \mathtt{u}_2^{2i-2} = \mathcal{P}_1(\mathtt{u}^{2i-2}), 
\\
& \mathtt{u}^{2i}  = \mathtt{u}_1^{2i-1} + \mathcal{L}_2(\mathtt{u}_1^{2i-1}) = \mathcal{P}_2(\mathtt{u}^{2i-1}).
\end{align}
% where $\mathcal{P}_1: H \rightarrow H$, and $\mathcal{P}_2: H \rightarrow H$ are the linear operators defined in \eqref{eq:P_1} and \eqref{eq:P_2}, respectively.

\begin{theorem}[convergence of \textbf{Algorithm} \ref{Algorithm_cont}]
Given $f_{1} \in L^{2}(\Omega_{1})$ and $f_{2} \in L^{2}(\Omega_{2})$, the sequence $\{ u^i \}_{i \geq 0}$ generated by the continuous alternating method described in \eqref{eq:schwartz_method} converges to the unique minimizer $u = u_{1} + u_{2} \in H$ of the energy $\mathcal{E}$ defined in \eqref{eq:energy_2}, which is characterized as the unique solution of the system \eqref{eq:first_EL_eq}--\eqref{eq:second_EL_eq}:
\begin{equation}
u_{1}^{i} \rightarrow u_{1} \text{ in } \tilde{H}^{s_{1}}(\Omega_{1}), \qquad u_{2}^{i} \rightarrow u_{2} \text{ in } \tilde{H}^{s_{2}}(\Omega_{2}), \qquad i \uparrow \infty.
\label{eq:convergence_plain}
\end{equation}
Moreover, the method is geometrically convergent: there exists $\kappa \in (0,1)$ such that
\begin{equation}
|u_{1} - u_{1}^{i}|_{\tilde{H}^{s_{1}}(\Omega_{1})} + |u_{2} - u_{2}^{i}|_{\tilde{H}^{s_{2}}(\Omega_{2})} \lesssim \kappa^{i} \qquad \forall i \in \mathbb{N}.
\label{eq:convergence_geometrical}
\end{equation}
\end{theorem}
\begin{proof}
Given the equivalence presented in Remark \ref{rem:conv_schwartz_method}, we assume that $f_1 \equiv 0$ and $f_2 \equiv 0$ and analyze the convergence of the sequence $\{ \mathtt{u}^i \}_{i \geq 0}$ described in \eqref{eq:schwartz_method_0_0}. Let us note that within this setting, i.e., $f_1 \equiv 0$ and $f_2 \equiv 0$,  the unique solution of the coupled problem \eqref{eq:first_EL_eq}--\eqref{eq:second_EL_eq} is $u \equiv 0 + 0$. 

Given the previously derived results (Lemma \ref{lemma:V1+V2} and Lemma \ref{lemma:orth_proj}) we are able to apply \cite[Theorem I.1]{Lions} and obtain convergence. In fact, $\mathcal{P}_i$, for $i\in \{1,2\}$, is an orthogonal projection from $H$ onto $V_i$ with respect to the inner product $\langle \cdot, \cdot \rangle_{\mathcal{A}}$ (cf. Lemma \ref{lemma:orth_proj}) and $H = V_1^{\intercal} \oplus V_2^{\intercal}$ (cf. Lemma \ref{lemma:V1+V2} combined with \cite[Corollary 2.15]{MR2759829}). We can thus obtain that there exists $\kappa \in (0,1)$ such that
\[
 \mathtt{u}^{i} \rightarrow 0 \textrm{ in } (H, | \cdot |_{\mathcal{A}}) \textrm{ as } i \uparrow \infty,
 \qquad
 | \mathtt{u}^{i+1} |_{\mathcal{A}} \leq \kappa^{i} | \mathtt{u}^{0} |_{\mathcal{A}} \quad \forall i \geq 0.
\]
From this we can deduce that $ \mathtt{u}_1^{i} \rightarrow 0$ in $\tilde{H}^{s_1}(\Omega_1)$ and that $ \mathtt{u}_2^{i} \rightarrow 0$ in $\tilde{H}^{s_2}(\Omega_2)$ as $i \uparrow \infty$. Let us now use the coercivity property \eqref{eq:A_is_coercive} of the bilinear form $\mathcal{A}$ to obtain
\[
|\mathtt{u}_1^{i+1}|_{\tilde{H}^{s_1}(\Omega_1)}^{2} + |\mathtt{u}_2^{i+1}|_{\tilde{H}^{s_2}(\Omega_2)}^{2} \lesssim |\mathtt{u}^{i+1}|_{\mathcal{A}}^{2} \lesssim \kappa^{2i}
\quad
\forall i \geq 0.
\]
This directly implies that $|\mathtt{u}_1^i|_{\tilde{H}^{s_1}(\Omega_1)} \lesssim \kappa^{i}$ and $|\mathtt{u}_2^i|_{\tilde{H}^{s_2}(\Omega_2)} \lesssim \kappa^{i}$ for all $i \geq 0$.
\end{proof}

\subsection{The discrete alternating scheme}

%%%%%%%%%%%%%%%%%%%%%%%%%%%%%%%%%%%%%%%%%%%%%%%%%%%%%%%
%%%%%%%%%%%%%%%%%  AlGORITHM 2    %%%%%%%%%%%%%%%%%%%%%
%%%%%%%%%%%%%%%%%%%%%%%%%%%%%%%%%%%%%%%%%%%%%%%%%%%%%%%

In this section, we present the discrete alternating \textbf{Algorithm} \ref{Algorithm_disc}, which is an iterative method for solving the finite element discretization \eqref{eq:first_EL_eq_discrete}--\eqref{eq:second_EL_eq_discrete}.

\begin{algorithm}[ht]
\caption{\textbf{The discrete alternating scheme}}
\label{Algorithm_disc}
\textbf{Input:} $\Omega_1, \Omega_2 \subset \mathbb{R}^n$, $s_{1},s_{2} \in (0,1)$, $f_{1} \in L^{2}(\Omega_{1})$, $f_{2} \in L^{2}(\Omega_{2})$, $J \in L^{1}(\mathbb{R}^{n})$, $\mathscr{T}_{1,h},\mathscr{T}_{2,\mathfrak{h}}$, and $u_{2,\mathfrak{h}}^{0} \in \mathbb{V}_{2,\mathfrak{h}}$.
\\
\textbf{Step 0:} Define $u_{h,\mathfrak{h}}^0 = 0 + u_{2,\mathfrak{h}}^{0} \in \mathbb{H}_{h,\mathfrak{h}}$.
\\
\textbf{For} $i=1$ \textbf{until convergence do}
\begin{itemize}
\item[$\boldsymbol{1}$:] Find the solution $u_{1,h}^{2i-1} \in \mathbb{V}_{1,h}$ of \eqref{eq:first_EL_eq_discrete} with $u_{2,\mathfrak{h}}$ replaced by $u_{2,\mathfrak{h}}^{2i-2}$.
\item[$\boldsymbol{2}$:] Define $u_{2,\mathfrak{h}}^{2i-1} =  u_{2,\mathfrak{h}}^{2i-2} \in \mathbb{V}_{2,\mathfrak{h}}$ and $u_{h,\mathfrak{h}}^{2i-1} =  u_{1,h}^{2i-1}  + u_{2,\mathfrak{h}}^{2i-1} \in \mathbb{H}_{h,\mathfrak{h}}$.
\item[$\boldsymbol{3}$:] Find the solution $u_{2,\mathfrak{h}}^{2i} \in \mathbb{V}_{2,\mathfrak{h}}$ of \eqref{eq:second_EL_eq_discrete} with $u_{1,h}$ replaced by $u_{1,h}^{2i-1}$.
\item[$\boldsymbol{4}$:] Define $u_{1,h}^{2i} =  u_{1,h}^{2i-1} \in \mathbb{V}_{1,h}$ and $u_{h,\mathfrak{h}}^{2i} =  u_{1,h}^{2i}  + u_{2,\mathfrak{h}}^{2i} \in \mathbb{H}_{h,\mathfrak{h}}$.
\end{itemize}
\end{algorithm}

According to the arguments developed in section \ref{sec:continuous_alternating_scheme}, the following results emerge.

\begin{theorem}[convergence of \textbf{Algorithm} \ref{Algorithm_disc}]
Given $f_{1} \in L^{2}(\Omega_{1})$ and $f_{2} \in L^{2}(\Omega_{2})$, the sequence $\{ u_{h,\mathfrak{h}}^i \}_{i \geq 0}$ generated by the discrete alternating method described in \textbf{Algorithm} \ref{Algorithm_disc} converges to the unique solution of the discrete coupled system \eqref{eq:first_EL_eq_discrete}--\eqref{eq:second_EL_eq_discrete}:
\begin{align*}
u_{1,h}^{i} \rightarrow u_{1,h} \text{ in } \mathbb{V}_{1,h}, \qquad u_{2,\mathfrak{h}}^{i} \rightarrow u_{2,\mathfrak{h}} \text{ in } \mathbb{V}_{2,\mathfrak{h}}, \qquad i \uparrow \infty.
\end{align*}
Moreover, the method is geometrically convergent: there exists $\kappa \in (0,1)$ such that
\begin{equation*}
|u_{1,h} - u_{1,h}^{i}|_{\tilde{H}^{s_{1}}(\Omega_{1})}
+
|u_{2,\mathfrak{h}} - u_{2,\mathfrak{h}}^{i}|_{\tilde{H}^{s_{2}}(\Omega_{2})} \lesssim \kappa^{i} \qquad \forall i \in \mathbb{N}.
\end{equation*}
\end{theorem}

%%%%%%%%%%%%%%%%%%%%%%%%%%%%%%%%%%%%%%%%%%%%%%%%%%%%
%%%%%%%%%%%%%%%%%%%%%%%%%%%%%%%%%%%%%%%%%%%%%%%%%%%%
%%%%%%%%%%%%%%%%%%%%%%%%%%%%%%%%%%%%%%%%%%%%%%%%%%%%
%%%%%%%%%%%%%%%%%%%%%%%%%%%%%%%%%%%%%%%%%%%%%%%%%%%%

\section{Numerical experiments}\label{sec:numerical_exp}

In this section, we present two numerical experiments in two dimensions that illustrate the performance of the devised finite element scheme. These experiments were conducted using a MATLAB code that adapts the solution technique from \cite{acosta2017short} for the integral fractional Laplacian to our framework, combined with the alternating method described in Algorithm~\ref{Algorithm_disc}. All integrals involved are approximated using a suitable Gaussian quadrature.

We now describe Algorithm~\ref{Algorithm_disc_fem}, which implements the discrete alternating scheme to approximate the solution of the discrete system \eqref{eq:first_EL_eq_discrete}--\eqref{eq:second_EL_eq_discrete}. Given two meshes, $\mathscr{T}_{1,h}$ for $\Omega_{1}$ and $\mathscr{T}_{2,\mathfrak{h}}$ for $\Omega_{2}$, the algorithm produces the approximation sequence $\{u_{h,\mathfrak{h}}^{k}\}_{k=0}^{N} \subset \mathbb{H}_{h,\mathfrak{h}}$, where $N$ depends on a prescribed tolerance criterion, and outputs
\[
u_{h,\mathfrak{h}}^{N} = u_{1,h}^{N} + u_{2,\mathfrak{h}}^{N} \in \mathbb{H}_{h,\mathfrak{h}}
\]
as the final approximation. The algorithm is initialized in \textbf{Step~0} with an initial guess  $u_{2,\mathfrak{h}}^0 \in \mathbb{V}_{2,\mathfrak{h}}$, setting
$u_{h,\mathfrak{h}}^0 = 0 + u_{2,\mathfrak{h}}^0 \in \mathbb{H}_{h,\mathfrak{h}}$,
where $0$ denotes the zero function on $\Omega_1$, $\mathtt{err} = 1$,
and $i = 1$. Each iteration $i \geq 1$ proceeds as follows. \textbf{Steps~1--4}
alternate between solving the discrete problem on $\Omega_1$ using the
current approximation on $\Omega_2$, and solving the discrete problem
on $\Omega_2$ using the updated approximation on $\Omega_1$, yielding
the new iterates $u_{h,\mathfrak{h}}^{2i-1}$ and $u_{h,\mathfrak{h}}^{2i}$.
\textbf{Step~5} then updates the error indicator
\[
\mathtt{err} = \|(u_{1,h}^{2i}, u_{2,\mathfrak{h}}^{2i}) -
(u_{1,h}^{2i-2}, u_{2,\mathfrak{h}}^{2i-2})\|_{2},
\]
where $\|\cdot\|_2$ denotes the Euclidean norm of the vector of nodal
values in both subdomains. The algorithm terminates when
$\mathtt{err} < 10^{-8}$.

\begin{algorithm}[ht]
\caption{\textbf{Implemented discrete alternating scheme}}
\label{Algorithm_disc_fem}
\textbf{Input:} $\mathscr{T}_{1,h},\mathscr{T}_{2,\mathfrak{h}}$, and $u_{2,\mathfrak{h}}^{0} \in \mathbb{V}_{2,\mathfrak{h}}$.
\\
\textbf{Step 0:} Define $u_{h,\mathfrak{h}}^0 = 0 + u_{2,\mathfrak{h}}^{0} \in \mathbb{H}_{h,\mathfrak{h}}$, where $0$ denotes the zero function on $\Omega_1$. Set $\mathtt{err} = 1$ and $i = 1$.
\\
\textbf{While} $\mathtt{err} \geq 10^{-8}$ \textbf{do}
\begin{itemize}
\item[$\boldsymbol{1}$:] Find the solution $u_{1,h}^{2i-1} \in \mathbb{V}_{1,h}$ of \eqref{eq:first_EL_eq_discrete} with $u_{2,\mathfrak{h}}$ replaced by $u_{2,\mathfrak{h}}^{2i-2}$.
\item[$\boldsymbol{2}$:] Define $u_{2,\mathfrak{h}}^{2i-1} =  u_{2,\mathfrak{h}}^{2i-2} \in \mathbb{V}_{2,\mathfrak{h}}$ and $u_{h,\mathfrak{h}}^{2i-1} =  u_{1,h}^{2i-1}  + u_{2,\mathfrak{h}}^{2i-1} \in \mathbb{H}_{h,\mathfrak{h}}$.
\item[$\boldsymbol{3}$:] Find the solution $u_{2,\mathfrak{h}}^{2i} \in \mathbb{V}_{2,\mathfrak{h}}$ of \eqref{eq:second_EL_eq_discrete} with $u_{1,h}$ replaced by $u_{1,h}^{2i-1}$.
\item[$\boldsymbol{4}$:] Define $u_{1,h}^{2i} =  u_{1,h}^{2i-1} \in \mathbb{V}_{1,h}$ and $u_{h,\mathfrak{h}}^{2i} =  u_{1,h}^{2i}  + u_{2,\mathfrak{h}}^{2i} \in \mathbb{H}_{h,\mathfrak{h}}$.
\item[$\boldsymbol{5}$:] Set $\mathtt{err} = \|(u_{1,h}^{2i},u_{2,\mathfrak{h}}^{2i}) - (u_{1,h}^{2i-2},u_{2,\mathfrak{h}}^{2i-2})\|_{2}$ and $i \leftarrow i+1$.
\end{itemize}
\end{algorithm}

We now describe the setting for our numerical experiments. We set $n = 2$, $\Omega_{1} = B_{1}(x_{1})$, and $\Omega_{2} = B_{1}(x_{2})$, where $x_{1} = -x_{2} = (1,1)$. Note that $d = \textrm{dist}(\Omega_1,\Omega_2) = 2 \sqrt{2} - 2 >0$. We choose the kernel
\[
 J = |B_{6}(0)|^{-1}\chi_{{B_{6}(0)}}.
\]
Note that $r = 6 > 2 \sqrt{2} - 2 =d$, so assumption \textup{(J3)}  is satisfied. The exact solution $\bar{u} = \bar{u}_1 +
\bar{u}_2$ is given by
\begin{align}\label{eq:exact_sol}
\begin{split}
& \bar{u}_{1}(x)= ( 2^{2s_{1}}\Gamma^{2}\left(1 + s_{1}\right))^{-1}(1 - |x - x_{1}|^{2})^{s_{1}}_{+}, 
\\
& \bar{u}_{2}(x)= ( 2^{2s_{2}}\Gamma^{2}\left(1 + s_{2}\right))^{-1}(1 - |x - x_{2}|^{2})^{s_{2}}_{+}, 
\end{split}
\end{align}
where $t_{+} = \max\{0,t\}$. Within this setting, following the approach of Theorem \ref{thm:Sob_reg}, the source terms \( f_1 \) and \( f_2 \) are given by
\begin{equation}
f_1 = 1 - \bar{u}_1\left( 2\mathcal{C}_{1} h_1 - g_1 \right) - \vartheta_1,
\qquad f_2 = 1 - \bar{u}_2\left( 2\mathcal{C}_{2} h_2 - g_2 \right) - \vartheta_2,
\label{eq:f_1_and_f_2}
\end{equation}
where 
\begin{align*}
& h_{1}(x) = \int_{\Omega_{2}} |x-y|^{-2 - 2s_{1}} \mathrm{d}y, \qquad h_{2}(x) = \int_{\Omega_{1}} |x-y|^{-2 - 2s_{2}} \mathrm{d}y, \\
& \vartheta_{1}(x) = \int_{\Omega_{2}}J(x - y) \bar{u}_{2}(y) \mathrm{d}y, \qquad \vartheta_{2}(x) = \int_{\Omega_{1}}J(x - y)\bar{u}_{1}(y) \mathrm{d}y,
\\
& g_{1}(x) = \int_{\Omega_{2}} J(x-y) \mathrm{d}y,
\qquad g_{2}(x) = \int_{\Omega_{1}} J(x-y) \mathrm{d}y.
\end{align*}
In \eqref{eq:f_1_and_f_2}, the constant $1$ follows from the fact that $(-\Delta)^{s_i} \bar{u}_i = 1$ in $\Omega_i$; see \cite[Theorem 3]{MR3640641} and \cite[Theorem 6.1]{acosta2017short}.

Finally, we note that in all experiments below, the alternating method converges in $4$ or $5$ iterations.

\subsection{Example 1: Convergence behavior for different fractional parameters}

We consider various combinations of $s_{1}$ and $s_{2}$, specifically
\[
(s_{1}, s_{2}) \in \{(0.2,0.2),(0.2,0.4),(0.2,0.6),(0.4,0.4),(0.4,0.6),(0.6,0.6)\}.
\]
The mesh sizes are set equal, $h = \mathfrak{h}$. This experiment illustrates the performance of the devised finite element method and provides experimental validation of the error estimate in Theorem  \ref{thm:apriori_error}, measured in the norm $| \cdot |_{H}$ defined in \eqref{eq:originalNorm}.

In Figure \ref{fig:ex_1}, we present the experimental convergence rates for $|\bar{u} - \bar{u}_{h,\mathfrak{h}}|_{H}$.
We observe that, for each combination of $s_{1}$ and $s_{2}$ considered, the experimental rate of convergence behaves as $O(h^{1/2})$. This agrees with the error bound in Theorem~\ref{thm:apriori_error} when $s_1,s_2 \geq 1/2$. However, when $s_{1} < 0.5$ or $s_{2} < 0.5$, the experimental rates of convergence are higher than those predicted
by Theorem \ref{thm:apriori_error}, but are consistent with the regularity result:
\begin{equation}\label{eq:higher_reg}
\bar{u}_i  \in H^{s_{i} + \frac{1}{2} - \epsilon_{i}}(\Omega_{i}), \qquad 0 < \epsilon_{i} < s_{i} + \tfrac{1}{2}, \qquad i \in \{1,2\}.
\end{equation}
The error bounds derived in Theorem \ref{thm:apriori_error} are based on the regularity estimates of Theorem \ref{thm:Sob_reg}, which follow from \cite[Theorem 2.1]{MR4283703}. For $s_{i} \in (0,0.5]$, this gives $\bar{u}_{i} \in H^{2s_{i} - \epsilon_{i}}(\Omega_{i})$ for all $\epsilon_{i} \in (0,s_{i})$, which is weaker than \eqref{eq:higher_reg}. As noted in \cite[page 1921]{MR4283703}, higher regularity is expected when the
forcing term belongs to $H^{r}(\Omega_{i})$ for some $r > 0$, but this cannot be derived from \cite[Theorem 2.1]{MR4283703}.
We emphasize that these regularity estimates hold under the sole assumption that $\partial \Omega_{i}$ is Lipschitz. Finally, we note that $\bar{u}_{1}$ and $\bar{u}_{2}$ defined in \eqref{eq:exact_sol} satisfy \eqref{eq:higher_reg}.

\begin{figure}[ht!]
\centering
\hspace{1.0cm}\includegraphics[trim={0 0 0 0},clip,width=9.00cm,height=6.0cm,scale=0.41]{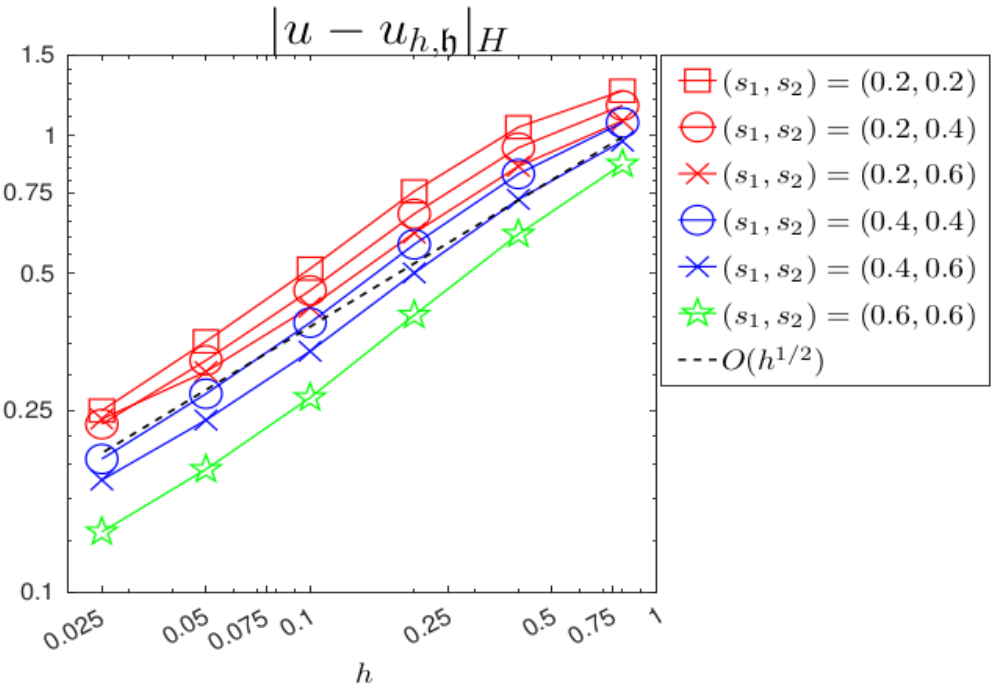}
\caption{Example 1: Experimental rates of convergence of $|\bar{u} - u_{h,\mathfrak{h}}|_{H}$ for various combinations of $s_1$
and $s_2$: $(s_{1}, s_{2}) \in \{(0.2,0.2),(0.2,0.4),(0.2,0.6),(0.4,0.4),(0.4,0.6),(0.6,0.6)\}.$}
\label{fig:ex_1}
\end{figure}

\subsection{Example 2: Influence of refining only the mesh 
$\mathscr{T}_{1,h}$}

In this example, we fix the mesh size $\mathfrak{h}$ associated with the mesh $\mathscr{T}_{2,\mathfrak{h}}$. We consider
$$(s_{1}, s_{2}) \in \{(0.2,0.2),(0.4,0.2),(0.6,0.2)\}.$$
The purpose of this experiment is to investigate the influence of refining only the mesh $\mathscr{T}_{1,h}$ on the errors $|u_{1} - u_{1,h}|_{\tilde{H}^{s_{1}}(\Omega_{1})}$ and $|u_{2} - u_{2,\mathfrak{h}}|_{\tilde{H}^{s_{2}}(\Omega_{2})}$.

In Figure~\ref{fig:ex_2}, we present the experimental rates of convergence for
\[
 |u_{1} - u_{1,h}|_{\tilde{H}^{s_{1}}(\Omega_{1})},
 \qquad
 |u_{2} - u_{2,\mathfrak{h}}|_{\tilde{H}^{s_{2}}(\Omega_{2})}.
\]
For all considered values of $(s_{1},s_{2})$, the error $|u_{1} - u_{1,h}|_{\tilde{H}^{s_{1}}(\Omega_{1})}$ initially exhibits an optimal experimental rate of convergence, which persists up to the fourth refinement step. After this point, the convergence rate deteriorates, as
expected: the fixed mesh $\mathscr{T}_{2,\mathfrak{h}}$
causes the error $|u_2 - u_{2,\mathfrak{h}}|_{\tilde{H}^{s_2}
(\Omega_2)}$ to stagnate, which in turn prevents further
decay of $|u_1 - u_{1,h}|_{\tilde{H}^{s_1}(\Omega_1)}$.

\begin{figure}[ht!]
\centering
\includegraphics[trim={0 0 0 0},clip,width=13.0cm,height=3.6cm,scale=0.37]{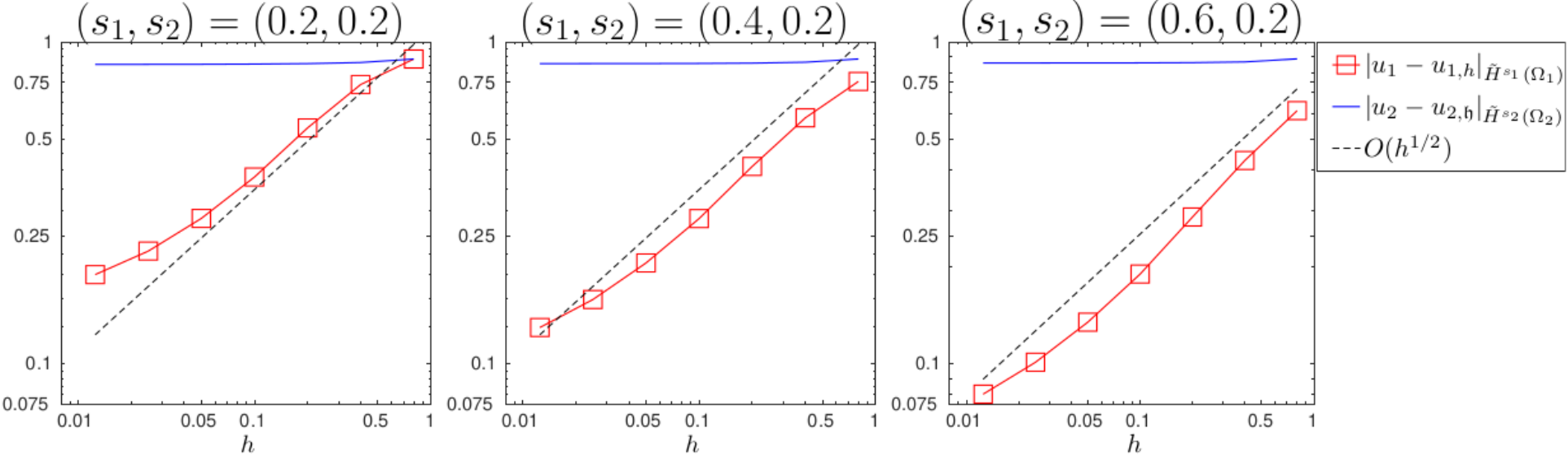}
\caption{Example 2: Experimental rates of convergence of $|u_{1} - u_{1,h}|_{\tilde{H}^{s_{1}}(\Omega_{1})}$ and $|u_{2} - u_{2,\mathfrak{h}}|_{\tilde{H}^{s_{2}}(\Omega_{2})}$ for various combinations of $s_1$ and $s_2$: $(s_{1}, s_{2}) \in \{(0.2,0.2),(0.4,0.2),(0.6,0.2)\}.$}
\label{fig:ex_2}
\end{figure}

\section{Conclusions}\label{sec:conclusions}

In this work, we introduced and analyzed a nonlocal coupled system involving two regional fractional Laplacians, each acting on a separate domain and coupled through a nonlocal interaction kernel. Well-posedness of the continuous problem was established via the direct method of the calculus of variations. We then proposed and analyzed a finite element discretization based on piecewise linear elements, for which we obtained a priori error estimates grounded in regularity theory for the fractional Laplacian on Lipschitz domains. The resulting convergence rates are in agreement with the observed experimental behavior.

To solve the discrete coupled system efficiently, we proposed and analyzed a discrete alternating scheme inspired by the classical Schwarz method. As an instrumental step, we first established geometric convergence of a continuous counterpart; the same arguments then yield geometric convergence of the discrete scheme to the unique solution of the discrete coupled system. The analysis rests on the orthogonal decomposition $H = V_1 \oplus V_2$ and the theory of alternating projections.

The numerical experiments confirm the theoretical predictions and reveal that, when $s_i < 1/2$, the experimental rates of convergence exceed those predicted by the theory. This phenomenon is consistent with higher regularity of the exact solution, which cannot currently be derived from available regularity theory for the fractional Laplacian on Lipschitz domains.

Natural directions for future work include the study of nonlocal coupled systems on non-disjoint domains, the analysis of more general coupling kernels,  the extension to nonlinear problems, and the formulation and analysis of optimal control problems governed by \eqref{def:state_eq}.

%---------------------------------
% BIBLIOGRAFIA
%---------------------------------

%\bibliographystyle{plain}

\bibliographystyle{siamplain}
\bibliography{nonlocal_coupled_ref}

\end{document}